\documentclass[11pt]{article}
\pdfoutput=1

\usepackage[letterpaper,margin=1in,bottom=1.4in]{geometry} 
\usepackage[format=plain,font=small,labelfont=bf,textfont=up]{caption}

\usepackage[utf8]{inputenc} 
\usepackage[T1]{fontenc}    
\usepackage[hidelinks]{hyperref}       
\usepackage{url}            
\usepackage{booktabs}       
\usepackage{amsfonts}       
\usepackage{nicefrac}       
\usepackage{microtype}      
\usepackage[affil-it]{authblk} 
\usepackage{tikz}
\usepackage{lmodern}
\usepackage{float}

\usepackage[linesnumbered,ruled,vlined]{algorithm2e}

\SetCommentSty{mycommfont}

\usepackage{amsmath}
\usepackage{amssymb}
\usepackage{amsthm}
\usepackage{mathtools}
\mathtoolsset{showonlyrefs,showmanualtags}
\allowdisplaybreaks
\usepackage{bbm}
\usepackage{graphicx}
\usepackage{subcaption}
\usepackage{thmtools, thm-restate}  
\usepackage[toc,page,header]{appendix}
\usepackage{minitoc}

\newcounter{spacesave}
\newcounter{journal}
\usepackage[natbibapa,nodoi]{apacite}
\DeclarePairedDelimiter\paren\lparen\rparen

\DeclarePairedDelimiter\braces\lbrace\rbrace

\DeclarePairedDelimiter\abs\lvert\rvert

\providecommand{\bbone}{\mathbf{1}}
\DeclarePairedDelimiterXPP\indicator[1]{\bbone}{\lbrack}{\rbrack}{}{#1}

\DeclarePairedDelimiterXPP\expf[1]{\exp}{\lparen}{\rparen}{}{#1}
\DeclarePairedDelimiterXPP\logf[1]{\log}{\lparen}{\rparen}{}{#1}
\DeclarePairedDelimiterXPP\maxf[1]{\max}{\lparen}{\rparen}{}{#1}
\DeclarePairedDelimiterXPP\minf[1]{\min}{\lparen}{\rparen}{}{#1}

\DeclareMathOperator*{\sgn}{sgn}
\DeclarePairedDelimiterXPP\sgnf[1]{\sgn}{\lparen}{\rparen}{}{#1}

\DeclarePairedDelimiterXPP\func[2]{#1}{\lparen}{\rparen}{}{#2}

\DeclareMathOperator*{\atan}{atan}
\DeclarePairedDelimiterXPP\atanf[1]{\atan}{\lparen}{\rparen}{}{#1}
\DeclarePairedDelimiterXPP\tanf[1]{\tan}{\lparen}{\rparen}{}{#1}

\DeclarePairedDelimiter\setb\lbrace\rbrace

\newcommand{\Reals}{\mathbb{R}}

\DeclarePairedDelimiter\card\lvert\rvert

\renewcommand{\vec}[1]{\boldsymbol{#1}}

\newcommand{\onevec}{\vec{1}}
\newcommand{\zerovec}{\vec{0}}

\makeatletter
\newcommand*{\tran}{{\mathpalette\@tran{}}}
\newcommand*{\@tran}[2]{\raisebox{\depth}{$\m@th#1\intercal$}}
\makeatother

\DeclarePairedDelimiter\norm\lVert\rVert
\DeclarePairedDelimiterXPP\tnorm[1]{}{\lVert}{\rVert_{1}}{}{#1}
\DeclarePairedDelimiterXPP\enorm[1]{}{\lVert}{\rVert_{2}}{}{#1}
\DeclarePairedDelimiterXPP\inorm[1]{}{\lVert}{\rVert_{\infty}}{}{#1}
\DeclarePairedDelimiterXPP\pnorm[2]{}{\lVert}{\rVert_{#1}}{}{#2}
\DeclarePairedDelimiterXPP\opnorm[1]{}{\lVert}{\rVert_{op}}{}{#1}

\DeclarePairedDelimiterXPP\detf[1]{\det}{\lparen}{\rparen}{}{#1}

\DeclarePairedDelimiterXPP\kerf[1]{\ker}{\lparen}{\rparen}{}{#1}

\DeclareMathOperator{\trsym}{tr}
\DeclarePairedDelimiterXPP\tr[1]{\trsym}{\lparen}{\rparen}{}{#1}

\DeclareMathOperator{\vspansym}{span}
\DeclarePairedDelimiterXPP\vspan[1]{\vspansym}{\lparen}{\rparen}{}{#1}

\DeclareMathOperator{\diagsym}{diag}
\DeclarePairedDelimiterXPP\diag[1]{\diagsym}{\lparen}{\rparen}{}{#1}

\DeclareMathOperator{\ranksym}{rank}
\DeclarePairedDelimiterXPP\rank[1]{\ranksym}{\lparen}{\rparen}{}{#1}

\DeclareMathOperator{\vectorizesym}{vec}
\DeclarePairedDelimiterXPP\vectorize[1]{\vectorizesym}{\lparen}{\rparen}{}{#1}

\DeclareMathOperator*{\esssup}{ess\,sup}
\DeclarePairedDelimiterXPP\esssupf[1]{\esssup}{\lparen}{\rparen}{}{#1}

\let\Prsym\Pr
\let\Pr\relax
\DeclarePairedDelimiterXPP\Pr[1]{\Prsym}{\lparen}{\rparen}{}{%
	#1}

\DeclarePairedDelimiterXPP\Prsub[2]{\Prsym_{#1}}{\lparen}{\rparen}{}{%
	#2}

\DeclareMathOperator{\Esym}{E}
\DeclarePairedDelimiterXPP\E[1]{\Esym}{\lbrack}{\rbrack}{}{%
	#1}

\DeclarePairedDelimiterXPP\Esub[2]{\Esym_{#1}}{\lbrack}{\rbrack}{}{%
	#2}

\DeclareMathOperator{\Varsym}{Var}
\DeclarePairedDelimiterXPP\Var[1]{\Varsym}{\lparen}{\rparen}{}{%
	#1}

\DeclarePairedDelimiterXPP\Varsub[2]{\Varsym_{#1}}{\lparen}{\rparen}{}{%
	#2}

\DeclarePairedDelimiterXPP\EstVar[1]{\widehat{\Varsym}}{\lparen}{\rparen}{}{%
	#1}

\DeclareMathOperator{\Covsym}{Cov}
\DeclarePairedDelimiterXPP\Cov[1]{\Covsym}{\lparen}{\rparen}{}{%
	#1}

\DeclarePairedDelimiterXPP\Covsub[2]{\Covsym_{#1}}{\lparen}{\rparen}{}{%
	#2}

\DeclareMathOperator{\Corrsym}{Corr}
\DeclarePairedDelimiterXPP\Corr[1]{\Corrsym}{\lparen}{\rparen}{}{%
	#1}

\makeatletter
\newcommand{\indep}{\protect\mathpalette{\protect\@indep}{\perp}}
\newcommand*{\@indep}[2]{\mathrel{\rlap{$#1#2$}\mkern3mu{#1#2}}}
\makeatother

\newcommand{\bigOsym}{\mathcal{O}}
\DeclarePairedDelimiterXPP\bigO[1]{\bigOsym}{\lparen}{\rparen}{}{#1}
\DeclarePairedDelimiterXPP\bigOt[1]{\widetilde{\bigOsym}}{\lparen}{\rparen}{}{#1}

\newcommand{\littleOsym}{o}
\DeclarePairedDelimiterXPP\littleO[1]{\littleOsym}{\lparen}{\rparen}{}{#1}

\newcommand{\bigOpsym}{\bigOsym_p}
\DeclarePairedDelimiterXPP\bigOp[1]{\bigOpsym}{\lparen}{\rparen}{}{#1}

\newcommand{\littleOpsym}{\littleOsym_p}
\DeclarePairedDelimiterXPP\littleOp[1]{\littleOpsym}{\lparen}{\rparen}{}{#1}

\newcommand{\bigOmegasym}{\Omega}
\DeclarePairedDelimiterXPP\bigOmega[1]{\bigOmegasym}{\lparen}{\rparen}{}{#1}

\newcommand{\littleOmegasym}{\omega}
\DeclarePairedDelimiterXPP\littleOmega[1]{\littleOmegasym}{\lparen}{\rparen}{}{#1}

\newcommand{\bigThetasym}{\Theta}
\DeclarePairedDelimiterXPP\bigTheta[1]{\bigThetasym}{\lparen}{\rparen}{}{#1}

\DeclarePairedDelimiterXPP\cosf[1]{\cos}{\lparen}{\rparen}{}{#1}
\DeclarePairedDelimiterXPP\sinf[1]{\sin}{\lparen}{\rparen}{}{#1}

\newcommand{\quadtext}[1]{\quad\text{#1}\quad}

\newcommand{\quadand}{\quadtext{and}}

\newcommand{\newvar}[2]{
	\expandafter\newcommand\csname #1\endcsname{#2}
}

\newcommand{\newvars}[2]{
	\expandafter\newcommand\csname #1\endcsname[1]{#2_{##1}}
}

\newcommand{\newvarss}[2]{
	\expandafter\newcommand\csname #1\endcsname[2]{#2_{##1,##2}}
}

\newcommand{\newvarsss}[2]{
	\expandafter\newcommand\csname #1\endcsname[3]{#2_{##1,##2,##3}}
}

\newcommand{\newvarssss}[2]{
	\expandafter\newcommand\csname #1\endcsname[4]{#2_{##1,##2,##3,##4}}
}

\newcommand{\newfunconly}[2]{
	\expandafter\DeclarePairedDelimiterXPP\csname #1\endcsname[1]{#2}{\lparen}{\rparen}{}{##1}
}

\newcommand{\newfuncs}[2]{
	\expandafter\DeclarePairedDelimiterXPP\csname #1\endcsname[2]{#2_{##1}}{\lparen}{\rparen}{}{##2}
}

\newcommand{\newfuncss}[2]{
	\expandafter\DeclarePairedDelimiterXPP\csname #1\endcsname[3]{#2_{##1,##2}}{\lparen}{\rparen}{}{##3}
}

\newcommand{\newfuncsss}[2]{
	\expandafter\DeclarePairedDelimiterXPP\csname #1\endcsname[4]{#2_{##1,##2,##3}}{\lparen}{\rparen}{}{##4}
}

\newcommand{\newfuncssss}[2]{
	\expandafter\DeclarePairedDelimiterXPP\csname #1\endcsname[5]{#2_{##1,##2,##3,##4}}{\lparen}{\rparen}{}{##5}
}

\newcommand{\varpartialeval}[3]{
	\expandafter\newcommand\csname #2\endcsname{\csname #1\endcsname{#3}}
}

\newcommand{\funcpartialeval}[3]{
	\expandafter\newcommand\csname #2\endcsname[1][]{\csname #1\endcsname[##1]{#3}}
}

\newcommand{\varevali}[2][]{
	\varpartialeval{#2#1}{#2i}{i}
	\varpartialeval{#2#1}{#2j}{j}
}

\newcommand{\varevalk}[2][]{
	\varpartialeval{#2#1}{#2k}{k}
	\varpartialeval{#2#1}{#2l}{\ell}
}

\newcommand{\varevals}[2][]{
	\varpartialeval{#2#1}{#2s}{s}
	\varpartialeval{#2#1}{#2r}{r}
}

\newcommand{\funcevali}[2][]{
	\funcpartialeval{#2#1}{#2i}{i}
	\funcpartialeval{#2#1}{#2j}{j}
}

\newcommand{\funcevalk}[2][]{
	\funcpartialeval{#2#1}{#2k}{k}
	\funcpartialeval{#2#1}{#2l}{\ell}
}

\newcommand{\varevalii}[2][]{
	\varevali[#1]{#2}
	\varevali{#2i}
	\varevali{#2j}
}

\newcommand{\varevalik}[2][]{
	\varevali[#1]{#2}
	\varevalk{#2i}
	\varevalk{#2j}
}

\newcommand{\varevaliik}[2][]{
	\varevalii[#1]{#2}
	\varevalk{#2ii}
	\varevalk{#2ij}
	\varevalk{#2ji}
	\varevalk{#2jj}
}

\newcommand{\varevaliis}[2][]{
	\varevalii[#1]{#2}
	\varevals{#2ii}
	\varevals{#2ij}
	\varevals{#2ji}
	\varevals{#2jj}
}

\newcommand{\varevaliikk}[2][]{
	\varevaliik[#1]{#2}
	\varevalk{#2iik}
	\varevalk{#2ijk}
	\varevalk{#2jik}
	\varevalk{#2jjk}
	\varevalk{#2iil}
	\varevalk{#2ijl}
	\varevalk{#2jil}
	\varevalk{#2jjl}
}

\newcommand{\funcevalii}[2][]{
	\funcevali[#1]{#2}
	\funcevali{#2i}
	\funcevali{#2j}
}

\newcommand{\funcevalik}[2][]{
	\funcevali[#1]{#2}
	\funcevalk{#2i}
	\funcevalk{#2j}
}

\newcommand{\funcevaliik}[2][]{
	\funcevalii[#1]{#2}
	\funcevalk{#2ii}
	\funcevalk{#2ij}
	\funcevalk{#2ji}
	\funcevalk{#2jj}
}

\newcommand{\funcevaliikk}[2][]{
	\funcevaliik[#1]{#2}
	\funcevalk{#2iik}
	\funcevalk{#2ijk}
	\funcevalk{#2jik}
	\funcevalk{#2jjk}
	\funcevalk{#2iil}
	\funcevalk{#2ijl}
	\funcevalk{#2jil}
	\funcevalk{#2jjl}
}

\newcommand{\funcevalX}[3]{
	\expandafter\newcommand\csname #1v#2\endcsname{\csname #1v\endcsname{#3}}
}

\newcommand{\funcevaliX}[3]{
	\expandafter\newcommand\csname #1v#2\endcsname{\csname #1v\endcsname{#3}}
	\expandafter\newcommand\csname #1v#2e\endcsname[1]{\csname #1ve\endcsname{##1}{#3}}
	\expandafter\newcommand\csname #1v#2i\endcsname{\csname #1vi\endcsname{#3}}
	\expandafter\newcommand\csname #1v#2j\endcsname{\csname #1vj\endcsname{#3}}
}

\newcommand{\funcevalikX}[3]{
	\expandafter\newcommand\csname #1v#2\endcsname{\csname #1v\endcsname{#3}}
	\expandafter\newcommand\csname #1v#2e\endcsname[2]{\csname #1ve\endcsname{##1}{##2}{#3}}
	\expandafter\newcommand\csname #1v#2i\endcsname[1]{\csname #1vi\endcsname{##1}{#3}}
	\expandafter\newcommand\csname #1v#2j\endcsname[1]{\csname #1vj\endcsname{##1}{#3}}
	\expandafter\newcommand\csname #1v#2ik\endcsname{\csname #1vik\endcsname{#3}}
	\expandafter\newcommand\csname #1v#2il\endcsname{\csname #1vil\endcsname{#3}}
	\expandafter\newcommand\csname #1v#2jk\endcsname{\csname #1vik\endcsname{#3}}
	\expandafter\newcommand\csname #1v#2jl\endcsname{\csname #1vjl\endcsname{#3}}
}

\newcommand{\newvari}[2]{
	\newvar{#1}{#2}
	\newvars{#1e}{\csname #1\endcsname}
	\varevali[e]{#1}
}

\newcommand{\newvark}[2]{
	\newvar{#1}{#2}
	\newvars{#1e}{\csname #1\endcsname}
	\varevalk[e]{#1}
}

\newcommand{\newvarik}[2]{
	\newvar{#1}{#2}
	\newvarss{#1e}{\csname #1\endcsname}
	\varevalik[e]{#1}
}

\newcommand{\newvarii}[2]{
	\newvar{#1}{#2}
	\newvarss{#1e}{\csname #1\endcsname}
	\varevalii[e]{#1}
}

\newcommand{\newvariik}[2]{
	\newvar{#1}{#2}
	\newvarsss{#1e}{\csname #1\endcsname}
	\varevaliik[e]{#1}
}

\newcommand{\newvariis}[2]{
	\newvar{#1}{#2}
	\newvarsss{#1e}{\csname #1\endcsname}
	\varevaliis[e]{#1}
}

\newcommand{\newvariikk}[2]{
	\newvar{#1}{#2}
	\newvarssss{#1e}{\csname #1\endcsname}
	\varevaliikk[e]{#1}
}

\newcommand{\newfunc}[2]{
	\newvar{#1}{#2}
	\newfunconly{#1v}{\csname #1\endcsname}
}

\newcommand{\newfunci}[2]{
	\newvari{#1}{#2}
	\newfunconly{#1v}{\csname #1\endcsname}
	\newfuncs{#1ve}{\csname #1\endcsname}
	\funcevali[e]{#1v}
}

\newcommand{\newfunck}[2]{
	\newvark{#1}{#2}
	\newfunconly{#1v}{\csname #1\endcsname}
	\newfuncs{#1ve}{\csname #1\endcsname}
	\funcevalk[e]{#1v}
}

\newcommand{\newfuncik}[2]{
	\newvarik{#1}{#2}
	\newfunconly{#1v}{\csname #1\endcsname}
	\newfuncss{#1ve}{\csname #1\endcsname}
	\funcevalik[e]{#1v}
}

\newcommand{\newfuncii}[2]{
	\newvarii{#1}{#2}
	\newfunconly{#1v}{\csname #1\endcsname}
	\newfuncss{#1ve}{\csname #1\endcsname}
	\funcevalii[e]{#1v}
}

\newcommand{\newfunciik}[2]{
	\newvariik{#1}{#2}
	\newfunconly{#1v}{\csname #1\endcsname}
	\newfuncsss{#1ve}{\csname #1\endcsname}
	\funcevaliik[e]{#1v}
}

\newcommand{\newfunciikk}[2]{
	\newvariikk{#1}{#2}
	\newfunconly{#1v}{\csname #1\endcsname}
	\newfuncssss{#1ve}{\csname #1\endcsname}
	\funcevaliikk[e]{#1v}
}

\theoremstyle{plain}
\newtheorem{theorem}{Theorem}[section]

\newtheorem{corollary}[theorem]{Corollary}
\newtheorem{lemma}[theorem]{Lemma}
\newtheorem{proposition}[theorem]{Proposition}

\theoremstyle{definition}

\newtheorem{definition}{Definition}

\theoremstyle{remark}

 \newcommand{\zv}{\vec{z}}
 \newcommand{\allinterventions}{\Omega}
 \newcommand{\events}{\mathcal{F}}
 \newcommand{\design}{\mathcal{D}}
 
 \newcommand{\yv}{\vec{y}}
 
 \newcommand{\modelspace}{\mathcal{M}}
 \newcommand{\nullmodel}{\modelspace_{0}}
 \newcommand{\altmodel}{\modelspace_{1}}
 
 \newcommand{\restrictmodel}{\modelspace}
 \newcommand{\nullhyp}{\restrictmodel_0}
 \newcommand{\althyp}{\restrictmodel_1}
 \newcommand{\althypd}{\althyp(\delta)}

\newcommand{\test}{\phi}

\newcommand{\expmapsym}{h}
\newcommand{\expmap}[1]{\expmapsym_{#1}}
\newcommand{\expset}[1]{\Delta_{#1}}
\newcommand{\allexpmap}{\vec{\expmapsym}}

\newcommand{\refmapsym}{m}
\newcommand{\refmap}[1]{\refmapsym_{#1}}
\newcommand{\splitset}[2]{s_{#1}(#2)}
\newcommand{\splitnum}[1]{S_{#1}}
\newcommand{\avgsplit}{S_{\textrm{avg}}}
 
 \newcommand{\neigh}[1]{\mathcal{N}(#1)}
 \newcommand{\neighi}{\neigh{i}}

 \newcommand{\degree}[1]{d_{#1}}
 \newcommand{\degreei}{\degree{i}}
 \newcommand{\degreej}{\degree{j}}
 \newcommand{\dmax}{\degree{\textrm{max}}}
 
 \newcommand{\risk}{\mathcal{R}}
 \newcommand{\optrisk}{\risk^*}
 
 \newcommand{\mix}[1]{\Pi_{#1}}
 \newcommand{\marg}[1]{Q_{#1}(\zv)}
 
 \newcommand{\tvdist}[2]{d_{\textrm{TV}}( #1 , #2 )}
 
 \newcommand{\sepest}{\widehat{g}}

\usepackage{expcmd}
\newcommand\mainref\ref
\newcommand\suppref\ref


\title{On the Impossibility of Specification Testing\\of Interference Models Based on Exposure Mappings}
\author[1]{Chao Gao}
\author[2]{Christopher Harshaw}
\author[3]{Fredrik S{\"a}vje}
\author[4]{Yitan Wang}
\affil[1]{University of Chicago}
\affil[2]{Columbia University}
\affil[3]{Uppsala University}
\affil[4]{Yale University}
\date{\today}

\usepackage{xcolor}         

\begin{document}
	
	\makeatletter%
	
	\begin{NoHyper}\gdef\@thefnmark{}\@footnotetext{\hspace{-1em}The authors gratefully acknowledge support from the following agencies:
Alfred Sloan fellowship;
NSF Grants ECCS-2216912, DMS-2310769, DMS2023505 and MMS-2316335;
Jan Wallander, Tom Hedelius \& Tore Browaldh foundations Grant P25-0067;
ONR Award N00014-20-1-2335;
Swedish Research Council Grant 2025-04794.
Part of this research was conducted at the Simons Institute for the Theory of Computing, and at the Isaac Newton Institute for Mathematical Sciences.
}\end{NoHyper}%
	\makeatother%
	
	\maketitle
	\thispagestyle{empty}
	
	\begin{abstract}
		Researchers use interference models based on exposure mappings to facilitate estimation of causal effects in randomized experiments with interference.
To test the veracity of such models, researchers can use specification tests that aim to detect departures from the stipulated model.
However, existing tests suffer from poor power and are often unable to detect important model violations.
The main result in this paper is to show that the specification testing problem for exposure mapping models is inherently difficult, and the poor power of existing tests is inescapable.
In particular, the worst-case Type I and Type II error rates must sum to one for any specification test of such models, ruling out the existence of a uniformly consistent test.
This is the worst-case overall error rate achieved by a naive test that discards all data and arbitrarily rejects the null at random; the testing problem is in this sense impossible.
This negative result holds true for all exposure mappings, all sample sizes, for uniformly bounded outcomes, and for alternatives that are maximally separated from the null.
While some tests can detect some type of departures from the null model, there will always be relevant departures from the null that are undetectable.
Informative specification tests must therefore restrict the alternative model against which they seek to attain power for, beyond the restrictions imposed by the exposure mappings alone.
We illustrate this by providing a uniformly consistent test for differentiating no-interference from a network-linear-in-means model.

	\end{abstract}

	\newpage
	
	\pagenumbering{roman}
	
	\doparttoc 
	\faketableofcontents 
	\part{} 
	\parttoc 
	
	\newpage

	\pagenumbering{arabic}
	
	\section{Introduction} \label{sec:intro}

Interference occurs when the treatment assigned to one unit causally affects other units in an experiment.
When interference is present, researchers typically stipulate a model of the structure of the interference, and estimate causal effects under the assumption that the specified model is correct.
The predominate way to specify an interference model in empirical practice is to use exposure mappings based on a network of historical interactions between units \citep{Aronow2017Estimating}.

It can be challenging to specify appropriate interference models, because the pattern of interference is often complex and subtle, and researchers sometimes worry that their models do not capture all interference.
A way to alleviate this worry is to conduct a specification test, examining whether the data collected from the experiment provides evidence that the stipulated model is misspecified.
Testing of interference models can also be of independent scientific interest.

In this paper, we examine testing of interference models.
Our main result is that uniformly informative specification testing of models based on exposure mappings is impossible, in the sense that there exists no uniformly consistent test of the correctness of any such models.
In fact, we show that the sum of the minimax optimal Type I and Type II error rates always sum to one, meaning that there exists no specification test that is better in the worst case than a test that discards all data and arbitrarily rejects the null at random.
This impossibility result holds for all exposure mappings, for all sample sizes, for uniformly bounded outcomes, and for alternatives that are maximally separated from the null.

Progress can be made only by making a priori restrictions on the structure of the interference beyond the restrictions imposed by the exposure mappings themselves.
That is, we must rule out some potential outcome functions by assumption without support from the data.
The problem is related to the fact that exact Fisher randomization tests have poor power against certain deptures from a sharp null hypothesis.
Our main result can be seen as a formalization and extension of this insight to all possible testing procedures and all possible interference models based on exposure mappings.
One consequence is that interference models based on exposure mappings cannot be validated using the experimentally collected data alone.

A conclusion of our main result is, however, not that tests of interference models are not without practical value.
Researchers who understand their proper scope can use specification tests of interference models to learn about the structure of interference in settings where it is reasonble to impose additional a priori restrictions on the interference structure beyond the restrictions imposed by the exposure mappings alone.
To illustrate this perspective, we provide a consistent specification test of the null hypothesis of no interference versus the alternative hypothesis that the interference is structured according to a network linear-in-means model \citep{Manski1993Identification,Bramoulle2009Identification}.
The network linear-in-means model is a restriction of a more general exposure mapping model, which is what allows us to construct an uniformly informative test.

\subsection{Related Works}

The main strand of the literature on testing for interference builds on Fisher-style randomization tests \citep{Fisher1935Design,Zhang2023What}.
Such tests require sharp null hypotheses, allowing all potential outcomes to be imputed for all (counterfactual) treatment assignments.
A common sharp null hypothesis is that treatment is completely inefficacious, including no interference effects, meaning that a test using such a null also implicitly tests for the presence of interference.
\citet{Bowers2013Reasoning} extends this logic to richer interference models, provided that the models imply sharp null hypotheses.

The main limitation of unconditional randomization tests is that the conventional sharp null restricts both the primary effects and interference, so it is not a test exclusively of interference.
A null hypothesis of no interference alone is not globally sharp, so some potential outcomes cannot be imputed and the full unconditional distribution of common test statistics cannot be calculated.
\citet{Aronow2012General} addresses this with a conditional test using a test statistic that depends only on the outcomes of a subset of the units, often called the focal units.
The distribution of the statistic is calculated conditional on the realized treatment assignments of the focal units, ensuring that all relevant potential outcomes can be imputed under the null.

Several papers have extended the conditional randomization test approach.
\citet{Athey2018Exact} study a broader class of null hypotheses related to interference and the choice of focal units.
\citet{Basse2019Randomization} provide a general framework for conditional randomization tests that allows more flexible conditioning on the observed treatments with the aim of improving power.
\citet{Puelz2022Graph} describe a graph-theoretic approach for finding a large set of focal units again with the aim of improving power.
\citet{Tiwari2024Quasi} use random graph models as a source of randomness, yielding better power in settings with highly dependent treatments.
To our knowledge, there are no formal analyses of power for conditional randomization tests, and power properties of these tests have primarily been investigated heuristically using simulations.

Specification tests for interference not based on Fisher-style randomization tests have also been developed.
\citet{Pouget2019Testing} describe a testing procedure using a hierarchical design, under which treatments are assigned using either complete or cluster-based randomization.
Without interference, the expected values of common treatment effect estimators are the same under both designs, so differences between such estimators can serve as a test of a no-interference assumption.
\citet{Choi2025Agnostic} describes a method to estimate a lower bound on the number of units affected by other units' treatments.
When there is no interference, there are no such units, so the lower bound can similarly act as a basis for a specification test.

The predominant method to specify interference models in the recent interference literature is to use exposure mappings, as introduced by \citet{Aronow2017Estimating}.
The related concept of effective treatments was introduced by \citet{Manski2013Identification}.
An exposure mapping is a low-dimensional summary of the full treatment vector.
An exposure mapping is said to be correctly specified when the summary it produces preserves all treatment information relevant for outcomes.
Most of this literature focuses on estimation and inference for average exposure effects under a correctly specified interference model rather than testing the correctness of the model itself.
However, \citet{Hoshino2023Randomization} adapt the conditional-randomization approach to test interference models based on exposure mappings, but the power properties of such specification tests have not been formally characterized.

Throughout the paper, we use the overall testing error to evaluate the performance of a specification test.
The overall testing error takes into account both Type I and Type II error rates of a test.
The minimax testing error of  a hypothesis testing problem is commonly used to characterize its difficulty \citep{Giné2021Mathematical}.

	\section{Design-Based Specification Tests} \label{sec:general-framework}

\subsection{Preliminaries}

We consider a randomized experiment with $n$ units.
The researcher implements an intervention $\zv$ chosen at random from a set $\allinterventions$, containing all interventions under consideration.
The experimental design is the (random) mechanism by which an intervention is chosen from $\allinterventions$.
Formally, the interventions form a measurable space $(\allinterventions, \events)$ and the experimental design $\design$ is a probability measure $\design: \events \to [0,1]$.
In conventional experiments with binary treatments, an intervention is a binary vector $\zv = (z_1, z_2, \dots z_n) \in \allinterventions = \setb{0,1}^n$.
Our results extend beyond this conventional setting, but we focus on binary treatments for clarity.

Each unit $i \in [n]$ has an associated \emph{potential outcome function} $y_i: \allinterventions \to \Reals$ that captures the unit's response under each intervention in $\allinterventions$.
The potential outcome functions depend on the entire intervention $\zv \in \allinterventions$, allowing for interference. 
For notational convenience, we collect the $n$ individual functions into a vector-valued function $\yv : \allinterventions \to \Reals^n$ that captures the responses for all units under each intervention:
\[
\yv(\zv) = \paren{ y_1(\zv), y_2(\zv), \dots y_n(\zv) } \enspace.
\]

Researchers impose structural restrictions on the potential outcome functions to facilitate estimation and inference \citep{Harshaw2022Design}.
We can encode the restrictions as subsets $\modelspace$ of the set of all measurable functions with signature $\allinterventions \to \Reals^n$.
We refer to such subsets as interference models.
A model $\modelspace$ is correctly specified if it contains the true potential outcome function: $\yv \in \modelspace$.
A common type of model in the recent interference literature is based on exposure mappings, which we describe in Section~\ref{sec:exposure-mapping-models}.


\subsection{Specification Tests and Minimax Testing Error} \label{sec:minimax-error}

We refer to the interference model under examination as the \emph{null model} and denote it with $\nullmodel$.
The null model is tested against an \emph{alternative model}, denoted $\altmodel$.
Intuitively, the null model consists of all potential outcome functions that the researcher deems plausible a priori in the empirical setting at hand.
The alternative model is more permissive, consisting of all functions that the researcher is not comfortable definitively ruling out a priori.
The null is thus nested in the alternative: $\nullmodel \subseteq \altmodel$.

In addition to structural interference restrictions, researchers often impose regularity conditions to ensure that the potential outcome functions are sufficiently well-behaved.
This could, for example, be moment conditions that prevent extreme outliers.
In this paper, we consider models that are restricted using a bound on the uniform moment: $\norm{\yv}_{\infty} = \max_{i\in [n]} \sup_{\zv \in \allinterventions} \abs{y_i(\zv)}$.
We use this moment condition because of its simplicity and because it is the strongest among commonly used moment bounds.
Proving impossibility under a strong moment condition implies impossibility also under weaker moment conditions.
To be precise, if $\nullmodel'$ and $\altmodel'$ denote the null and alternative models only encoding structural interference restrictions, we consider the following restricted null and alternative models:
\[
\nullmodel = \setb[\big]{ \yv \in \nullmodel' : \norm{\yv}_{\infty} \leq 1 }
\quadand
\altmodel = \setb[\big]{ \yv \in \altmodel' : \norm{\yv}_{\infty} \leq 1 }
\enspace.
\]
The choice of $1$ for the upper bound is purely to simplify the notation; it can be replaced by $\norm{\yv}_{\infty} \leq C$ for any constant $C > 0$ without any meaningful changes to our results.

A specification test seeks evidence against the null hypothesis $\yv \in \nullmodel$.
If sufficient evidence is found, the null model is rejected in favor of the alternative.
Formally, a \emph{statistical test} is a measurable function $\test : \allinterventions \times \Reals^n \to \setb{0,1}$ that takes as input the realized intervention $\zv$ and observed outcomes $\yv(\zv)$ and returns $1$ when the null is rejected.
The test can implicitly depend on other aspects of the units, such as covariates and observed historical patterns of interactions.
In the design-based framework, the researcher can select the experimental design in addition to the statistical test.
A \emph{testing procedure} is a pair $(\design, \test)$ for the specification testing problem at hand, consisting of an experimental design $\design$ and a statistical test $\test$.

\begin{definition}[Type I error control]
A testing procedure $(\design, \test)$ controls the Type I error at rate $\alpha \in [0, 1]$ if $\sup_{\yv \in \nullmodel} \Esub{\zv \sim \design}{ \test( \zv, \yv(\zv) ) } \leq \alpha$.
\end{definition}

Existing tests of interference models have predominately focused on controlling the Type I error.
However, Type I error control alone is not difficult to achieve nor particularly helpful.
Indeed, it is trivially achieved by a test that never rejects the null.
More subtly, it can be achieved by testing procedures that largely ignore the realized outcomes $\yv(\zv)$, and it might not be obvious from a cursorily inspection of a procedure that it achieves error control by effectively discarding most of the outcome data.
The concern with such procedures is that they lack power: the rejection rate $\Esub{\zv \sim \design}{ \test( \zv, \yv(\zv) ) }$ is low also in the alternative model $\yv \in \altmodel$.
If the rejection rate is $\alpha$ in both $\nullmodel$ and $\altmodel$, the test provides no information about whether the null model is correctly specified even if it successfully controls the Type I error rate.

For a testing procedure to be useful, it must have power against meaningful departures from the null model.
That is, the rejection rates must be notably larger than $\alpha$ for relevant functions in the alternative.
It is typically not possible to achieve large rejection rates for all functions in $\altmodel \setminus \nullmodel$, because functions in $\altmodel$ that are close to the boundary of $\nullmodel$ will be practically indistinguishable from some functions in $\nullmodel$.
Instead, we consider rejection rates for functions in $\altmodel$ that are well-separated from the null, in the sense that the functions are non-negligibly different from all functions in $\nullmodel$.
We formalize separation using a functional.

\begin{definition} \label{def:separation-functional}
	A \emph{separation functional} $g : \altmodel \to \Reals_{\geq 0}$ has two properties: 
	(i) (distinguishing) $g(\yv) = 0$ if and only if $\yv \in \nullmodel$
	(ii) (homogeneity) $g( \gamma \cdot \yv ) = \gamma^2 \cdot g(\yv)$ for all $\gamma \in \Reals$ and $\yv \in \altmodel$.
\end{definition}

The first property states that the functional distinguishes functions inside and outside the null model by assigning positive values only to functions that are outside the null.
The second property states that scaling a potential outcome function results in quadratically growing separation.
This ensures that the separation $g(\yv)$ grows as $\yv$ goes further away from the null.
It is possible to impose a degree of homogeneity other than two, but all separation functionals we consider will be quadratic.

A \emph{$\delta$-separated alternative model} $\altmodel(\delta)$ is the subset of functions $\yv \in \altmodel$ that are sufficiently different from the null according to the separation functional: $\altmodel(\delta) = \setb[\big]{\yv \in \altmodel : g(\yv) \geq \delta}$.
When evaluating a testing procedure, we consider alternatives that are $\delta$-separated from the null.
We refer to $\delta$ as the \emph{separation value}.
As $\delta$ increases, the corresponding alternative $\altmodel(\delta)$ becomes smaller, making the testing problem easier: $\altmodel(\delta) \subseteq \altmodel(\delta')$ for $\delta \geq \delta'$.

\begin{definition}\label{def:testing-error}
	Consider a testing procedure $(\design, \test)$.
	For a given separation value $\delta \geq 0$, the \emph{overall testing error} $\risk(\design, \test, \delta)$ is the sum of the worst-case Type I and Type II errors:
$$
\risk(\design, \test, \delta)
= \sup_{\yv \in \nullmodel} \Esub[\Big]{\zv \sim \design}{ \test( \zv, \yv(\zv) ) }
+ \sup_{\yv \in \altmodel(\delta)} \Esub[\Big]{\zv \sim \design}{ 1 - \test( \zv, \yv(\zv) ) }.
$$
\end{definition}

The overall testing error $\risk(\design, \test, \delta)$ captures how well a testing procedure can distinguish the null from well-separated functions in the alternative.
A testing procedure that is able to perfectly distinguish between the null model and a well-separated alternative has testing error $\risk(\design, \test, \delta) = 0$.
Testing procedures that discard the data and arbitrary reject the null at random, and procedures that never reject the null, have testing error $\risk(\design, \test, \delta) = 1$.
This is because any decrease in the Type I error rate is perfectly offset by a corresponding increase in the Type II error rate for such tests, indicating that they are uninformative about whether the null model is correctly specified.
In principle, the overall testing error can be greater than one, but practically, $\risk(\design, \test, \delta) = 1$ marks the greatest relevant error rate, as it can be achieved by naive procedures that ignore all data from the experiment.

\begin{definition}
A testing procedure $(\design, \test)$ is \emph{uniformly consistent} if the overall testing error approaches zero for all separation values: $\lim_{n \to \infty} \risk(\design, \test, \delta) \to 0$ for all fixed $\delta > 0$.
\end{definition}

A uniformly consistent testing procedure can distinguish between the null and any meaningful departures from the null given enough data.
Consistency is defined with respect to a sequence of null and alternative models, using triangular array asymptotics that is standard in the design-based literature. 
In this asymptotic regime, each experiment in the sequence has its own potential outcome function $\yv^{(n)}$, null model $\nullmodel^{(n)}$, and alternative model $\altmodel^{(n)}$, although we suppress the dependence on $n$ for notational simplicity.

We consider uniform consistency to be the minimal requirement for a testing procedure to be useful as a specification test in practice.
An inconsistent procedure would not be able to distinguish some meaningful departures from the null even if we had infinite data.
To use an inconsistent procedure for specification testing, we would be required to rule out a priori all meaningful departures from the null that the procedure is unable to detect.
But this defeats the purpose of the specification test, because the alternative model contains exactly the potential outcome functions that we were not comfortable to rule out a priori.

When we fail to construct a testing procedure that is uniformly consistent, or otherwise fail to achieve acceptable error rates, we might ask if this is due to our own inability or if the specification testing problem at hand is inherently difficult.
To understand the fundamental difficulty of the specification testing problem, we turn to the minimax testing error $\optrisk$ which represents the best achievable testing error of any testing procedure.

\begin{definition}\label{def:minimax-testing-error}
The \emph{minimax testing error} for separation $\delta$ is $\optrisk(\delta) = \inf_{\design, \test} \risk(\design, \test, \delta)$.
\end{definition}

\begin{definition}\label{def:impossible}
A specification testing problem is \emph{impossible} at separation $\delta$ if the corresponding minimax testing error is one: $\optrisk(\delta) = 1$.
\end{definition}

If a testing problem has $\optrisk(\delta) = 1$, there exists no testing procedure that is more informative about the correctness of the null, as judged by the overall testing error, than a naive procedure that ignores all information provided by the experiment.
In this sense, a testing problem is impossible when $\optrisk(\delta) = 1$.
There might exist testing procedures for impossible testing problems that can detect some types of departures from the null, and such a procedure would be informative for that particular type of misspecification of the null model.
However, if a testing problem is impossible, there will always be meaningful departures from the null that the researcher was not comfortable ruling out a priori but which are empirically undetectable.
Progress can then be made only by imposing additional restrictions on the null or alternative models that are neither substantiated by prior knowledge about the empirical setting nor supported by the data.

The main result of this paper, as described in the next section, is to show that specification tests of null and alternative models based on discrete exposure mappings are impossible in the sense formalized by Definition~\ref{def:impossible}.

	\section{Impossibility of Specification Testing of Exposure Mappings} \label{sec:impossibility-larger}

\subsection{Exposure Models} \label{sec:exposure-mapping-models}

Following \citet{Aronow2017Estimating}, an exposure mapping is a low-dimensional summary of the entire high-dimensional intervention $\zv \in \allinterventions$.
For each unit $i \in [n]$, the researcher specifies a finite set of exposures $\expset{i}$, and the exposure mapping $\expmap{i} : \allinterventions \to \expset{i}$ maps each intervention $\zv$ to an exposure in $\expset{i}$.
It is convenient to collect all $n$ exposure mappings into a single vector-valued function $\allexpmap(\zv) = ( \expmap{1}(\zv), \dotsc, \expmap{n}(\zv) )$.

An exposure model stipulates that the structure of interference is captured by some specified exposure mappings, in the sense that the dimension reduction from $\allinterventions$ to $\expset{i}$ is without loss of information for each unit.
Formally, the interference model based on a collection of exposure mappings $\allexpmap$ is 
\[
\modelspace_{\allexpmap} = \setb[\Big]{
	\yv :
	y_i(\zv) = y_i(\zv')
	\text{ for all }
	i \in [n]
	\text{ and }
	\zv, \zv'
	\text{ satisfying }
	\expmap{i}(\zv) = \expmap{i}(\zv')
	\text{ and }
	\norm{\yv}{\infty} \leq 1
}
\enspace.
\]
If $\modelspace_{\allexpmap}$ is correctly specified, then there exist functions $\tilde{y}_i : \expset{i} \to \Reals$ such that $y_i(\zv) = \tilde{y}_i(\expmap{i}(\zv))$ for all $i \in [n]$ and all $\zv \in \allinterventions$.
This formalizes the understanding that the dimension reduction imposed by correctly specified exposure mappings is without loss of information.
A consequence is that an interference model $\modelspace_{\allexpmap}$ based on exposure mappings can be parameterized by the collections of functions $\tilde{y}_1, \dotsc, \tilde{y}_n$, specifying the value of the true potential outcome function for all units and exposures.

Most interference models considered in the literature are explicitly or implicitly based exposure mapping models.
We here provide some common examples.

\begin{itemize}
	\item \textbf{No Treatment Effect.}
	Consider the model stipulating that the intervention has no causal effect whatsoever.
	In this case, each unit has a single exposure $\expset{i} = \setb{1}$ and the exposure mapping is a constant function: $\expmap{i}(\zv) = 1$.
	This model is correctly specified when the potential outcome function is constant, meaning that exists a $ \alpha_{i} \in \Reals$ for each unit such that $y_i(\zv) = \alpha_{i}$ for all $\zv \in \allinterventions$.
	This model is commonly used as the null hypothesis in Fisher-style randomization tests.
	
	\item \textbf{No Interference.}
	The most common interference model in causal inference is one that allows for arbitrary effects of the units' own treatments but no interference between units, commonly called the Stable Unit Treatment Value Assumption (SUTVA).
	Each unit has here two exposures $\expset{i} = \setb{0,1}$ and the mapping captures the unit's own treatment: $\expmap{i}(\zv) = z_i$.
	If correctly specified, we can write $y_i(\zv) = \alpha_{i,0} \indicator{z_i = 0} + \alpha_{i,1} \indicator{z_i =1 }$ for some $\alpha_{i,0}, \alpha_{i,1} \in \Reals$.
	
	\item \textbf{Stratified Network Interference.}
	A common interference model that allows for interference restricts it to be anonymous, in the sense that the number of treated neighboring units in some network matters for the outcome but not their identity \citep{Hudgens2008Toward}.
	To formalize this model, let $\neighi$ be the set of neighbors of unit $i$ in the network at hand.
	The exposure set of unit $i$ is then $\expset{i} = \setb{0,1} \times \setb{0, 1, \dotsc, \card{\neighi}}$, where $\card{\neighi}$ is the network degree of the unit.
	The corresponding exposure mapping captures the unit's own treatment and the number of treated neighboring units: $\expmap{i}(\zv) = \textstyle \paren[\big]{ z_i, \sum_{j \in \neighi} z_j }$.
	If correctly specified, the potential outcome function can take a most $2(\card{\neighi} + 1)$ unique values.
	
	\item \textbf{Arbitrary Network Interference.}
	A more complex interference model allows for non-anonymous interference, in the sense that the outcome might differ depending which neighbors are treated even if the number of treated neighbors is the same \citep{Kandiros2024Conflict}.
	Here, the exposure set of unit $i$ is the power set of its own index and the indices of its neighbors: $\expset{i} = \mathcal{P}(\setb{i} \cup \neighi)$.
	The exposure mapping captures the exact combination of treatments assigned to the unit and its neighbors:
	$\expmap{i}(\zv) = \setb{ j \in \setb{i} \cup \neighi : z_j = 1 }$.
	If correctly specified, the potential outcome function can take a most $2^{\card{\neighi} + 1}$ unique values, demonstrating that it allows for considerably more complex interference than the stratified interference model.
\end{itemize}

\subsection{Separation of Exposure Models}

We consider specification tests of a null model $\nullmodel$ based on a collection of exposure mappings $\allexpmap_0(\zv) = ( \expmap{1}^0(\zv), \expmap{2}^0(\zv), \dotsc, \expmap{n}^0(\zv) )$ against an alternative model $\altmodel$ based on some other exposure mappings $\allexpmap_1(\zv) = ( \expmap{1}^1(\zv), \expmap{2}^1(\zv), \dotsc, \expmap{n}^1(\zv) )$.
As in the general setting in Section~\ref{sec:general-framework}, the alternative is more permissive or expressive than the null, meaning that all functions in the null are also in the alternative: $\nullmodel \subseteq \altmodel$.
The examples in the previous subsection have such nested structure, where the each model is nested in all subsequent models.
For example, the third model, stratified interference, is a superset of the two preceding models and a subset of the fourth model.

Our results are not sensitive to the exact choice of the separation functional $g$, and we can accommodate essentially any reasonable choice.
For concreteness, we consider a functional based on the sum of maximum square deviations of $y_i$ in the alternative model among all interventions that map to the same exposure in the null model.
That is, if under the null $\expmap{i}^0(\zv) = \expmap{i}^0(\zv')$, then the null stipulates that $y_i(\zv) = y_i(\zv')$.
When the alternative departs from the null, there exists some $\yv \in \altmodel$ such that $y_i(\zv) \neq y_i(\zv')$.
Our separation functional captures the magnitude of this departure: $\braces{y_i(\zv) - y_i(\zv')}^2$.

Formally, the separation functional we use is
\[
g(\yv) = \frac{1}{n} \sum_{i=1}^n \sum_{e \in \expset{i}^0} \sup \setb[\Big]{ \braces[\big]{y_i(\zv) - y_i(\zv')}^2 : \zv, \zv' \in \allinterventions \;\text{ s.t. }\; \expmap{i}^0(\zv) = \expmap{i}^0(\zv') = e }.
\]
As intended, the functional is zero $g(\yv) = 0$ only when $\yv \in \nullmodel$.
The uniform moment bound $\abs{y_i(\zv)} \leq 1$ restricts how much functions in the alternative can depart from the null.
Among $\yv$ satisfying $\norm{\yv}_{\infty} \leq 1$, the separation can be at most $g(\yv) \leq (4 / n) \sum_{i=1}^n \card{\expset{i}^0}$ which corresponds to potential outcome function that differ from the null with the maximum amount $\braces{y_i(\zv) - y_i(\zv')}^2 = 4$ for all units and exposures.

Because the moment bound restricts the maximum separation, we have $\althypd = \emptyset$ for sufficiently large $\delta$.
If $\althypd = \emptyset$, the researcher has effectively ruled out all meaningful departures from the null already before running the experiment, making it pointless to test the null.
Thus, a non-trivial specification test necessitates a non-empty $\althypd$.
We refer to all $\delta$ that yield $\althypd = \emptyset$ as \emph{trivial separation values}.

\subsection{Impossibility of Specification Testing}

For a null model $\nullmodel$ based on exposure mappings and a well-separated alternative model $\altmodel(\delta)$ also based on exposure mappings, let $\optrisk(\delta)$ denote the minimax testing error according to Definition~\ref{def:minimax-testing-error} of the corresponding testing problem.

\expcommand{\generalimpossibility}{%
	The specification testing problem corresponding to a null exposure model $\nullmodel$ and the alternative exposure model $\althypd$ is impossible for all non-trivial separation values $\delta$, i.e. $\optrisk(\delta) = 1$.
}
\begin{theorem}\label{theorem:general-impossibility}
\generalimpossibility
\end{theorem}

Theorem~\ref{theorem:general-impossibility} says that specification testing of exposure mappings is an extremely challenging testing problem.
The theorem holds for all possible nested exposure models no matter how large the separation is between them, provided that the separation value $\delta$ remains non-trivial, and it holds no matter the number of units $n$ there are in the experiment.
A consequence of the theorem is that a uniformly consistent testing procedure can never be constructed for specification testing of exposure models.
Relaxing the uniform moment bound in $\nullhyp$ and $\althypd$ only makes the testing problem more difficult.

\subsection{Proof Sketch}

The proof of Theorem~\ref{theorem:general-impossibility} constructs a lower bound on the minimax testing error using mixture distributions $\mix{0}$ and $\mix{1}$ over the null model $\nullhyp$ and separated alternative model $\althypd$, respectively.
The intuition is that the testing problem is hard if we can construct $\mix{0}$ and $\mix{1}$ that make the observed data distributions similar under both mixtures.
The following lemma formalizes this intuition.

\expcommand{\lecamstep}{
	Given mixture $\mix{0}$ supported on the null $\nullhyp$ and mixture $\mix{1}$ supported on a non-empty $\delta$-separated alternative $\althypd$, the minimax testing error is lower bounded as 
	\[
	\optrisk(\delta) \geq 
	1 - \sup_{\design} \Esub[\big]{\zv \sim \design}{ \tvdist{\marg{0}}{\marg{1}} },
	\]
	where $\marg{k}$ is the distribution of $\yv(\zv) = (y_1(\zv), y_2(\zv), \dots y_n(\zv))$ under mixture $\mix{k}$ for a fixed $\zv$, and $d_{\textrm{TV}}$ is the total variation distance.
}
\begin{lemma}\label{lemma:mixture-lower-bound}
	\lecamstep
\end{lemma}

\begin{proof}
The lemma follows from two key insights, following an approach reminiscent of Le Cam's method for constructing lower bounds.
The first insight is that a maximum is lower bounded by an average.
In particular, for any mixture $\mix{0}$ supported on the null $\nullhyp$, we have
\begin{equation}
\sup_{\yv \in \nullhyp} \Esub[\big]{\zv \sim \design}{ \phi( \zv, \yv(\zv) ) } \geq \Esub[\Big]{ \yv \sim \mix{0}}{ \Esub[\big]{\zv \sim \design}{ \phi( \zv, \yv(\zv) ) } }.
\end{equation}
Applying this insight also to the alternative model and rearranging the order of the expectations,
\begin{equation}
		\optrisk(\delta)
\geq \inf_{\design} \inf_{ \phi } \setb[\bigg]{  
			\Esub[\Big]{\zv \sim \design}{
				\Esub[\big]{ \yv \sim \mix{0}}{ \phi( \zv, \yv(\zv) ) }
				+ \Esub[\big]{ \yv \sim \mix{1} }{ 1 - \phi( \zv, \yv(\zv)) }
			}
}.
\end{equation}
The second insight is that the infimum over tests $\phi$ is separable with respect to the interventions $\zv$, in the sense that we can interpret $\phi$ as set of $\card{\allinterventions}$ unrelated functions indexed by $\zv \in \allinterventions$.
For a fixed intervention $\zv$, let $\phi_{\zv}$ denote the function $\vec{u} \mapsto \phi(\zv, \vec{u})$.
After applying separability and rearranging, the right-hand side above equals
\begin{align*}
& = \inf_{\design} \Esub[\Bigg]{\zv \sim \design}{
	\inf_{ \phi_{\zv} : \Reals^n \to \setb{0,1} }
	\paren[\Big]{
		\Esub[\big]{\yv \sim \mix{0}}{ \phi_{\zv}(\yv(\zv)) }
		+ \Esub[\big]{ \yv \sim \mix{1}}{ 1 - \phi_{\zv}(\yv(\zv)) }
	}
}
\\
& = 1 - \sup_{\design} \Esub[\Bigg]{\zv \sim \design}{
	\sup_{ \phi_{\zv} : \Reals^n \to \setb{0,1} }
	\paren[\Big]{
		\Esub[\big]{ \yv \sim \mix{1}}{ \phi_{\zv}(\yv(\zv)) }
		- \Esub[\big]{\yv \sim \mix{0}}{ \phi_{\zv}(\yv(\zv)) }
	}
}.
\end{align*}
The lemma now follows by applying the definition of the total variation distance as the supremum over all measurable 0/1-valued functions of their expected difference under the two measures.
\end{proof}

The proof of Theorem~\ref{theorem:general-impossibility} proceeds by constructing explicit mixtures $\mix{0}$ and $\mix{1}$ for the null and alternative models with the property $\tvdist{ \marg{0} }{ \marg{1} } = 0$ for all interventions $\zv \in \allinterventions$.
This means that $\mix{0}$ and $\mix{1}$ are indistinguishable no matter the intervention, so no design can tell the two distributions apart.
The consequence is that $\sup_{\design} \Esub[\big]{\zv \sim \design}{ \tvdist{\marg{0}}{\marg{1}} } = 0$, which in turn implies that $\optrisk(\delta) \geq 1$ by Lemma~\ref{lemma:mixture-lower-bound}, completing the proof.

The construction of the mixtures $\mix{0}$ and $\mix{1}$ for the general case is conceptually simple but notationaly cumbersome, so we have opted to present the general construction in the 
\ifnum \value{journal}=1 {online supplement.} \else {appendix.} \fi
However, it is possible to illustrate the central aspects of the construction in a simpler setting, focusing on testing the sharp null of no treatment effect whatsoever against the alternative of no interference, corresponding to the first two examples described at the end of Section~\ref{sec:exposure-mapping-models}.
We state this setting as a corollary of Theorem~\ref{theorem:general-impossibility} for clarity.

\begin{corollary} \label{corr:fisher-impossibility}
Let the null exposure model be no treatment effects, using exposure mappings $\expmap{i}(\zv) = 1$.
Let the alternative exposure model be no interference, using exposure mappings $\expmap{i}(\zv) = z_i$.
The specification testing problem of distinguishing between these two models is impossible for all non-trivial separation values.
\end{corollary}

Potential outcome functions in the null model $\nullmodel$ are all of the form $y_i(\zv) = \alpha_i$, with $\alpha_i \in [-1, 1]$ by the moment restriction.
The null model can thus be parametrized by $\boldsymbol{\alpha} = \paren{\alpha_1, \dotsc, \alpha_n} \in [-1, 1]^n$.
Similarly, all functions in the alternative $\altmodel$ are of the form $y_i(\zv) = \beta_{i,0} \indicator{z_i = 0} + \beta_{i,1} \indicator{z_i = 1}$, so the alternative can be parametrized by $\boldsymbol{\beta}_0 = \paren{\beta_{1,0}, \dotsc, \beta_{n,0}}$ and $\boldsymbol{\beta}_1 = \paren{\beta_{1,1}, \dotsc, \beta_{n,1}}$, both in $[-1, 1]^{n}$.
The separation functional takes a simple form: $g(\yv) = n^{-1} \sum_{i=1}^n \paren[\big]{ \beta_{i,1} - \beta_{i,0}}^2$.

The mixture $\mix{0}$ we use for the null model is such that $\alpha_i$ is set to either $-1$ or $1$ with equal probability independently between units, meaning that that $\boldsymbol{\alpha}$ is uniform over $\setb{-1, 1}^n$.
The mixture $\mix{1}$ for the alternative model is such that $\boldsymbol{\beta}_0$ is also uniform over $\setb{-1, 1}^n$, but $\boldsymbol{\beta}_1 = - \boldsymbol{\beta}_0$.
This means that $\paren{ \beta_{i,1} - \beta_{i,0}}^2 = 4$ with probability one under $\mix{1}$, so all units always have the largest treatment effect allowed by the moment restriction.
Thus, we achieve the largest possible separation, and these mixtures can be used for any non-trivial separation value.

\begin{lemma} \label{lemma:sutva-mixtures}
The mixtures $\mix{0}$ and $\mix{1}$ are empirically indistinguishable: $\tvdist{\marg{0}}{\marg{1}} = 0$ for all $\zv \in \allinterventions$.
\end{lemma}

\begin{proof}
Fix an intervention $\zv \in \allinterventions$.
Under the null mixture $\mix{0}$, the distribution of $\yv(\zv) = \paren{\alpha_1, \dotsc, \alpha_n}$ is uniform over $\setb{-1, 1}^n$.
Under the alternative mixture $\mix{1}$, we have $\yv(\zv)
= \paren{\beta_{1,z_1}, \dotsc, \beta_{n,z_n}}$.
Marginally, each coordinate is uniform over $\setb{-1, 1}$, and because the parameters in $\boldsymbol{\beta}_0$ are independent between units, the coordinates of $\paren{\beta_{1,z_1}, \dotsc, \beta_{n,z_n}}$ are independent.
It follows that $\yv(\zv)$ is uniform over $\setb{-1, 1}^n$ also under $\mix{1}$.
Thus, $\marg{0} = \marg{1}$, so $\tvdist{\marg{0}}{\marg{1}} = 0$.
\end{proof}

	\section{Consistent Test Against Linear-in-Means Model} \label{sec:lim-test}

The impossibility result in the previous section applies to specification tests based on exposure mappings.
Uniformly informative specification tests are potentially possible for interference models that impose restrictions beyond those imposed by the exposure mappings alone.
To illustrate this, we consider testing a null exposure model against an alternative model that impose additional restrictions, effectively ruling out some of the potential outcome functions a priori.

The null $\nullmodel$ under consideration here is the exposure model based on mappings $\expmap{i}(\zv) = z_i$, capturing a no interference assumption.
The alternative is the so-called network linear-in-means model, which stipulates that each potential outcome function is linear in the fraction of treated neighbors in some observed network.
As above, let $\neighi$ be the set of unit $i$'s neighbors in an undirected network.
The linear-in-means model is then
\[
\altmodel = \setb[\big]{ 
	\yv : y_i(\zv) = \beta_{i,1} + \beta_{i,2} z_i + \beta_{i,3} T_i(\zv) 
	\text{ and }
	\norm{\yv}_{\infty} \leq 1
}
\enspace,
\]
where $T_i(\zv) = \card{\neighi}^{-1} \sum_{j \in \neighi} z_j$ is the fraction of unit $i$'s neighbors who are treated.
The degree of each unit is $\degreei = \abs{\neighi}$ and the maximum degree is $\dmax = \max_{i \in [n]} \degreei$.
In order to avoid minor technical details, we assume that each subject in the network has at least two neighbors: $\degreei \geq 2$.
We relax this assumption in the
\ifnum \value{journal}=1 {online supplement.} \else {appendix.} \fi
The separation functional is $g(\yv) = n^{-1} \sum_{i=1}^n \beta_{i,3}^2$, mirroring that $\yv \in \modelspace_{0}$ if and only if $\beta_{i,3} = 0$ for all $i \in [n]$.

The linear-in-means model is reminiscent of the stratified network interference model, presented as the third example in Section~\ref{sec:exposure-mapping-models}.
However, because the linear-in-means model stipulates that the potential outcome functions are linear in $T_i(\zv)$, it restricts the functions beyond the restrictions imposed by the exposures in the stratified model.
The linear-in-means model is therefore not an exposure model as understood in this paper.

We construct a uniformly consistent testing procedure by constructing a consistent estimator of the separation functional $g(\yv)$ evaluated at the true potential outcome function.
The design $\design$ is the Bernoulli design under which each treatment $z_i$ is drawn independently with $\Pr{z_i = 1} = p$.
We write $y_i$ and $T_i$ as shorthand for $y_i(\zv)$ and $T_i(\zv)$ to avoid notational clutter.
Our estimator of the separation value $g(\yv)$ is then $\sepest = n^{-1} \sum_{i = 1}^n W_i Y_i^2 $, where $Y_i = y_i(\zv)$ is the observed outcome and
\[
W_i = \frac{ \Var{T_i} (T_i^2 - \E{ T_i^2 }) - \paren[\big]{ \E{T_i^3} - \E{T_i^2} \E{T_i} } \paren[\big]{ T_i - \E{T_i} }  }{ \Var{T_i} (\E{ T_i^4 } - \E{ T_i^2 }^2) - \paren[\big]{ \E{T_i^3} - \E{T_i^2} \E{T_i} }^2 }.
\]
The weighting term $W_i$ depends only on $T_i$, whose distribution is known.
The separation estimator is based on the estimation approach for quadratic functionals described by \citet{Harshaw2022Design}, which was there used for variance estimation of linear point estimators.
The underlying idea of the approach is to construct $W_i$ so that $\Esub{\zv \sim \design}{ W_i Y_i^2 } = \beta_{i,3}^2$ for all $i \in [n]$.
We describe the construction of the estimator in detail in the 
\ifnum \value{journal}=1 {online supplement.} \else {appendix.} \fi

\begin{proposition}\label{prop:sep-estimator-mse}
	There is a $C > 0$ such that $\E{ \paren{\sepest - g(\yv)}^2  } \leq C \cdot \dmax^5 / n$ for all $\yv \in \altmodel$.
\end{proposition}

The proposition bounds the mean square error of the separation estimator using the maximum degree of the network $\dmax$ and is therefore consistent in mean square if $\dmax = \littleO{n^{1/5}}$.
To test the null $\nullmodel$ against a $\delta$-separated alternative $\altmodel(\delta)$, we consider a threshold test $\phi = \indicator{ \sepest \geq \tau }$ where the threshold is chosen as $\tau = (\dmax^{5} / n)^{1/4}$.
This choice ensures that the threshold will be converging to zero while also asymptotically dominating the statistical error of $\sepest$.
The following theorem shows that this procedure is uniformly consistent.

\begin{theorem}\label{theorem:lim-consistency}
	Consider the specification test of the no-interference model against the linear-in-means model.
	If $\dmax = \littleO{n^{1/5}}$, the Bernoulli design together and threshold test $\test = \indicator{\sepest \geq \tau}$ with $\tau = (\dmax^{5} / n)^{1/4}$ are uniformly consistent: $\lim_{n \to \infty} \risk(\design, \test, \delta) = 0$ for all $\delta > 0$.
\end{theorem}

The proofs of Proposition~\ref{prop:sep-estimator-mse} and Theorem~\ref{theorem:lim-consistency} are presented in the 
\ifnum \value{journal}=1 {online supplement.} \else {appendix.} \fi

	\section{Concluding remarks} \label{sec:conclusion}

An implication of our impossibility result is that specification testing is less useful than what we believe many researchers expect.
The most pessimistic perspective is that the scope of a specification testing is so narrow that it should fill essentially no role in empirical practice.
While we find the use of specification tests in current empirical practice too optimistic, we still see a role for specification tests, provided that researchers understand their limitations.
In particular, researchers should be mindful that there exists no test that can detect all types of departures from the null for specification testing problems based on exposure models.
Researchers must therefore rule out some departures from the null using prior knowledge and use testing procedures that target the departures that they are not comfortable ruling out a priori.
If a researcher is not comfortable ruling out any of the departures a priori, the only option is to entertain the possibility that the null is misspecified.

An important strand of the literature on interference has focused on testing of interference models using Fisher-style randomization tests and sharp null hypotheses.
It is acknowledged that this type of test has poor power against certain departures from the null.
Our impossibility result shows that poor power is an inescapable aspect of this type of testing problem.
The upshot is that the pursuit of uniformly informative specification tests of exposure mapping models is not productive, and focus should be directed to constructing tests that can detect certain types of departures.
For many existing tests of interference models, the relevant departures for which the test has power against are implicit.
It is therefore difficult for empirical researchers to judge whether the test is useful for their purposes.
We believe it would be useful to explicitly state what departures are of interest, and construct tests that target those departures.
The uniformly consistent test we describe in Section~\ref{sec:lim-test} is one such example.


	\bibliography{references.bib}

@article{Athey2018Exact,
	author = {Susan Athey and Dean Eckles and Guido W. Imbens},
	title = {Exact p-Values for Network Interference},
	journal = {Journal of the American Statistical Association},
	volume = {113},
	number = {521},
	pages = {230--240},
	year = {2018},
	doi = {10.1080/01621459.2016.1241178},
}

@Article{Aronow2012General,
	author   = {P. M. Aronow},
	title    = {A General Method for Detecting Interference between Units in Randomized Experiments},
	year     = {2012},
	journal  = {Sociological Methods {\&} Research},
	volume   = {41},
	number   = {1},
	pages    = {3--16},
}

@Article{Aronow2017Estimating,
	author   = {P. M. Aronow and Cyrus Samii},
	title    = {Estimating Average Causal Effects under General Interference},
	year     = {2017},
	journal  = {Annals of Applied Statistics},
	volume   = {11},
	number   = {4},
	pages    = {1912--1947},
}

@article{Basse2019Randomization,
	author = {Basse, G and Feller, A and Toulis, P},
	title = {Randomization tests of causal effects under interference},
	journal = {Biometrika},
	volume = {106},
	number = {2},
	pages = {487--494},
	year = {2019},
	doi = {10.1093/biomet/asy072},
}

@article{Bowers2013Reasoning,
	title={Reasoning about Interference Between Units: A General Framework},
	volume={21},
	DOI={10.1093/pan/mps038},
	number={1},
	journal={Political Analysis},
	author={Bowers, Jake and Fredrickson, Mark M. and Panagopoulos, Costas},
	year={2013},
	pages={97--124}
}

@article{Bramoulle2009Identification,
	title = {Identification of peer effects through social networks},
	journal = {Journal of Econometrics},
	volume = {150},
	number = {1},
	pages = {41--55},
	year = {2009},
	doi = {10.1016/j.jeconom.2008.12.021},
	author = {Yann Bramoullé and Habiba Djebbari and Bernard Fortin},
}

@misc{Choi2025Agnostic,
	title={Agnostic Characterization of Interference in Randomized Experiments}, 
	author={David Choi},
	year        = {2025},
	eprinttype  = {arXiv},
	eprint      = {2410.13142},
	eprintclass = {stat.ME},
	note        = {arXiv:2410.13142}
}

@book{Fisher1935Design,
	author = {R. A. Fisher},
	title = {The Design of Experiments},
	publisher = {Oliver and Boyd},
	year = {1935},
}

@book{Giné2021Mathematical, 
	place={Cambridge},
	series={Cambridge Series in Statistical and Probabilistic Mathematics},
	title={Mathematical Foundations of Infinite-Dimensional Statistical Models}, publisher={Cambridge University Press},
	author={Giné, Evarist and Nickl, Richard},
	year={2021},
	collection={Cambridge Series in Statistical and Probabilistic Mathematics}
}

@Unpublished{Harshaw2022Design,
	author      = {Christopher Harshaw and Fredrik S{\"a}vje and Yitan Wang},
	title       = {A General Design-Based Framework and Estimator for Randomized Experiments},
	year        = {2022},
	eprinttype  = {arXiv},
	eprint      = {2210.08698},
	eprintclass = {stat.ME},
	note        = {arXiv:2210.08698}
}

@misc{Hoshino2023Randomization,
	title={Randomization Test for the Specification of Interference Structure}, 
	author={Tadao Hoshino and Takahide Yanagi},
	year        = {2023},
	eprinttype  = {arXiv},
	eprint      = {2301.05580},
	eprintclass = {stat.ME},
	note        = {arXiv:2301.05580}
}

@article{Hudgens2008Toward,
	author = {Michael G Hudgens and M. Elizabeth Halloran},
	title = {Toward Causal Inference With Interference},
	journal = {Journal of the American Statistical Association},
	volume = {103},
	number = {482},
	pages = {832--842},
	year = {2008},
	doi = {10.1198/016214508000000292},
}

@Unpublished{Kandiros2024Conflict,
	author      = {Vardis Kandiros and Charilaos Pipis and Constantinos Daskalakis and Christopher Harshaw},
	title       = {The Conflict Graph Design: Estimating Causal Effects under Arbitrary Neighborhood Interference},
	year        = {2024},
	eprinttype  = {arXiv},
	eprint      = {2411.10908},
	eprintclass = {stat.ME},
	note        = {arXiv:2411.10908}
}

@article{Manski1993Identification,
	author = {Manski, Charles F.},
	title = {Identification of Endogenous Social Effects: The Reflection Problem},
	journal = {The Review of Economic Studies},
	volume = {60},
	number = {3},
	pages = {531--542},
	year = {1993},
	doi = {10.2307/2298123},
}

@article{Manski2013Identification,
	author = {Charles F. Manski},
	journal = {The Econometrics Journal},
	number = {1},
	pages = {S1--S23},
	title = {Identification of treatment response with social interactions},
	volume = {16},
	year = {2013}
}

@article{Pouget2019Testing,
	author = {Pouget-Abadie, J and Saint-Jacques, G and Saveski, M and Duan, W and Ghosh, S and Xu, Y and Airoldi, E M},
	title = {Testing for arbitrary interference on experimentation platforms},
	journal = {Biometrika},
	volume = {106},
	number = {4},
	pages = {929--940},
	year = {2019},
	doi = {10.1093/biomet/asz047},
}

@article{Puelz2022Graph,
	author = {Puelz, David and Basse, Guillaume and Feller, Avi and Toulis, Panos},
	title = {A Graph-Theoretic Approach to Randomization Tests of Causal Effects under General Interference},
	journal = {Journal of the Royal Statistical Society Series B: Statistical Methodology},
	volume = {84},
	number = {1},
	pages = {174--204},
	year = {2022},
	doi = {10.1111/rssb.12478},
}

@Unpublished{Tiwari2024Quasi,
	author      = {Supriya Tiwari and Pallavi Basu},
	title       = {Quasi-randomization tests for network interference},
	year        = {2024},
	eprinttype  = {arXiv},
	eprint      = {2403.16673},
	eprintclass = {stat.ME},
	note        = {arXiv:2403.16673}
}

@article{Zhang2023What,
	author = {Yao Zhang and Qingyuan Zhao},
	title = {What is a Randomization Test?},
	journal = {Journal of the American Statistical Association},
	volume = {118},
	number = {544},
	pages = {2928--2942},
	year = {2023},
	doi = {10.1080/01621459.2023.2199814},
}

	\newpage 
	\appendix
	
	\addcontentsline{toc}{section}{Appendix} 
	\part{Appendix} 
	\parttoc 
	
	\section{General Impossibility Result} \label{sec:supp-impossibility}

Let $\allexpmap_0$ and $\allexpmap_1$ be the exposure mappings that define models $\nullmodel$ and  $\altmodel$, respectively.
Each function in the null model $\yv \in \nullmodel$ can be expressed in a functional form using the exposure mapping $\allexpmap_0$ as
\[
y_i(\zv) = \sum_{e_0 \in \expset{i}^0} \alpha_{i,e_0} \indicator{ \expmap{i}^0(\zv) = e_0 }
\enspace.
\]
where $\alpha_{i, e_0}$ are coefficients for each subject $i \in [n]$ and exposure $e_0 \in \expset{i}^0$.
Similarly, each function in the alternative model $\yv \in \altmodel$ can be expressed in the functional form using the alternative exposure mapping $\allexpmap_1$ as
\[
y_i(\zv) = \sum_{e_1 \in \expset{i}^1} \beta_{i,e_1} \indicator{ \expmap{i}^1(\zv) = e_1 }
\enspace,
\]
where $\beta_{i, e_1}$ are coefficients for each subject $i \in [n]$ and exposure $e_1 \in \expset{i}^1$.
For notational clarity, we reserve the $\alpha$ coefficients to describe functions in the null model and reserve $\beta$ coefficients to describe functions in the alternative model.

\subsection{Exposure Mapping Refinements}

We are interested in exposure mappings $\allexpmap_0$ and $\allexpmap_1$ which define nested models $\nullmodel \subseteq \altmodel$.
In what follows, we show that nested exposure mappings can be understood in terms of refinements.
This perspective will allow for a clearer presentation of the general impossibility result Theorem~\mainref{theorem:general-impossibility}.
Fix a subject $i \in [n]$ and consider its two exposure mappings $\expmap{i}^0 : \allinterventions \to \expset{i}^0$ and $\expmap{i}^1 : \allinterventions \to \expset{i}^1$.
We say that $\expmap{i}^1$ is a \emph{refinement} of $\expmap{i}^0$ if there exists a mapping $\refmap{i} : \expset{i}^1 \to \expset{i}^0$ such that $\expmap{i}^0 = \refmap{i} \circ \expmap{i}^1$, i.e.
\[
\expmap{i}^0(\zv) = \refmap{1} \paren[\big]{ \expmap{i}^1(\zv) }
\enspace.
\]
Roughly speaking, $\expmap{i}^1$ is a refinement of $\expmap{i}^0$ when the exposure received by subject $i$ under $\expmap{i}^1$ completely determines the exposure received by subject $i$ under $\expmap{i}^0$.
In this case, we refer to $\refmap{i}$ as the \emph{refinement mapping} which maps the finer exposure set $\expset{i}^1$ to the coarser exposure set $\expset{i}^0$.
Given a coarse exposure $e_0 \in \expset{i}^0$, define its \emph{split set} $\splitset{i}{e_0}$ as the set of finer exposures $e_1 \in \expset{i}^1$ that map to $e_0$ under the refinement mapping, i.e.
\[
\splitset{i}{e_0} = \setb[\Big]{ e_1 \in \expset{i}^1 : \refmap{i}(e_1) = e_0 }
\enspace.
\]
A coarse exposure $e_0 \in \expset{i}^0$ is said to be \emph{split} if $\abs{\splitset{i}{e_0} } > 1$.
For each subject, we denote the the number of split exposures as $\splitnum{i} = \sum_{e_0 \in \expset{i}^0} \indicator{ \abs{\splitset{i}{e_0} } > 1  }$.
The average number of split exposures across all units is denoted $\avgsplit = n^{-1} \sum_{i=1}^n \splitnum{i}$.

\begin{proposition} \label{prop:refinement}
	Consider two exposure mappings $\allexpmap_0$ and $\allexpmap_1$ which define the null $\nullmodel$ and alternative $\altmodel$, respectively.
	$\nullmodel \subseteq \altmodel$ if and only if $\expmap{i}^1$ is a refinement of $\expmap{i}^0$ for all subjects $i \in [n]$.
\end{proposition}

Proposition~\ref{prop:refinement} shows that nested exposure mapping models can be equivalently understood in terms of refinement mappings.
The additional notation of refinement mapping will clarify the key ideas throughout.

Using the functional form of exposure mapping models together with the refinement mapping, the separation functional $g: \altmodel \to \Reals_{\geq0}$ defined in Section~\mainref{sec:minimax-error} may be expressed as
\[
g(\yv) 
= \frac{1}{n} \sum_{i=1}^n \sum_{e_0 \in \expset{1}^0} 
	\sup \braces[\Big]{ (\beta_{i,e_1} - \beta_{i,e_1'} )^2 : e_1, e_1' \in  \splitset{i}{e_0} }
\enspace.
\]
If $g(\yv) = 0$, then each subject's response $y_i(\zv)$ is constant across the split set $\splitset{i}{e_0}$ of every coarse exposure $e_0 \in \expset{i}^0$.
This means that $\yv \in \nullmodel$.
Likewise, if $g(\yv) > 0$ then there must be a subject whose response $y_i(\zv)$ varies across some split set $\splitset{i}{e_0}$, in which case $\yv \notin \nullmodel$.
We maintain $\yv \in \altmodel$ throughout. 

\subsection{Proof of Theorem~\mainref{theorem:general-impossibility}}

We are now ready to prove the general impossibility theorem.
For completeness, we recall the statement of the theorem below:

\renewcommand\thetheorem{\mainref{theorem:general-impossibility}}
\begin{theorem}
	\generalimpossibility
\end{theorem}
\renewcommand\thetheorem{\text{S}\arabic{theorem}}
\addtocounter{theorem}{-1}

Throughout the rest of the section, we fix an arbitrary choice of nested exposure mapping models $\nullmodel \subseteq \altmodel$ corresponding to exposure mappings $\allexpmap_0$ and $\allexpmap_1$.
In order to show the impossibility result of Theorem~\mainref{theorem:general-impossibility}, we rely on Lemma~\mainref{lemma:mixture-lower-bound} which lower bounds the minimax testing error in terms of mixtures $\mix{1}$ and $\mix{0}$.
We restate the lemma here for completeness:

\renewcommand\thelemma{\mainref{lemma:mixture-lower-bound}}
\begin{lemma}
	\lecamstep
\end{lemma}
\renewcommand\thelemma{\text{S}\arabic{lemma}}
\addtocounter{lemma}{-1}

Recall that functions $\yv \in \nullmodel$ and $\vec{u} \in \altmodel$ can be parameterized by
\[
y_i(\zv) = \sum_{e_0 \in \expset{i}^0} \alpha_{i,e_0} \indicator{ \expmap{i}^0(\zv) = e_0 }
\quadand
u_i(\zv) = \sum_{e_1 \in \expset{i}^1} \beta_{i,e_1} \indicator{ \expmap{i}^1(\zv) = e_1 }
\enspace,
\]
so we can equivalently describe the mixtures $\mix{0}$ and $\mix{1}$ in terms of the coefficients $\alpha_{i,e_0}$ and $\beta_{i,e_1}$, respectively.

The null mixture $\mix{0}$ is simple: all coefficients $\alpha_{i,e_0}$ are chosen independently and uniformly from $\setb{+1, -1}$.
In the alternative mixture $\mix{1}$, we will select coefficients $\beta_{i, e_1}$ which are still independent across subjects $i$, but will have dependence among exposures belonging to the same split set.
Observe that the split sets $\setb{ \splitset{i}{e_0} : e_0 \in \expset{i}^0 }$ form a partition of the finer exposure set $\expset{i}^1$.
Fix a coarse exposure $e_0 \in \expset{i}^0$ and suppose its split set $\splitset{i}{e_0}$ has size $k_{e_0}$.
Consider a $k_{e_0}$-length vector $\vec{v}_{e_0}$ with the following properties:
\begin{itemize}
	\item $\vec{v}_{e_0}$ has entries in $\setb{+1, -1}$.
	\item If $k_{e_0} > 1$, then $\vec{v}_{e_0}$ has at least one $+1$ entry and one $-1$ entry.
\end{itemize}
The alternative mixture $\mix{1}$ is now constructed as follows: independently for each coarse element $e_0 \in \expset{i}^0$, select the coefficients in the split set $\setb{ \beta_{i, e_1} : e_1 \in \splitset{i}{e_0} }$ to be $\vec{v}_{e_0}$ or $-\vec{v}_{e_0}$ with equal probability.

\begin{lemma} \label{lemma:mixture-support}
	The mixtures $\mix{0}$ and $\mix{1}$ satisfy the two conditions:
	\begin{enumerate}
		\item $\mix{1}$ is supported on $\nullhyp$
		\item $\mix{0}$ is supported on $\althypd$ for all non-trivial separation values $\delta$.
	\end{enumerate}
\end{lemma}
\begin{proof}
	By construction, functions in the support of $\mix{0}$ and $\mix{1}$ are bounded as $\norm{\yv}_{\infty} \leq 1$.
	Again by construction, the functions in the support of $\mix{0}$ and $\mix{1}$ have the correct structural form to be in the exposure mapping models $\nullmodel$ and $\altmodel$, respectively.
	
	Now, we seek to show that $\mix{1}$ is supported on the separated alternative model $\althypd$ for any non-trivial separation value $\delta$.
	To this end, observe that the separation value of any $\yv$ in the support of $\mix{1}$ is equal to $g(\yv) = \avgsplit$.
	\begin{align*}
	g(\yv)  
	&= \frac{1}{n} \sum_{i=1}^n \sum_{e_0 \in \expset{1}^0} 
	\sup \braces[\Big]{ (\beta_{i,e_1} - \beta_{i,e_1'} )^2 : e_1, e_1' \in  \splitset{i}{e_0} }
	\\
	&= \frac{1}{n} \sum_{i=1}^n \sum_{e_0 \in \expset{i}^0}  4 \cdot \indicator{ \abs{\splitset{i}{e_0}} > 1 } \\
	&\triangleq 4 \cdot \avgsplit
	\end{align*}
	The construction of the vector $\vec{v}_{e_0}$ ensures that the second equality holds; namely, that when the coarse exposure $e_0$ is split, there will be two finer exposures in its split set whose coefficients are $1$ and $-1$.
	The final equality follows from the definition of $\avgsplit$.
	
	By the arguments above, the support of $\mix{1}$ is contained in the separated alternative $\althypd$  for every $\delta \leq 4 \cdot \avgsplit$.
	We will now show that any larger separation value $\delta > 4  \cdot \avgsplit$ is trivial in the sense that it results in $\althypd$ being empty.
	To see this, observe that for any $\yv$ with $\norm{\yv}_{\infty} \leq 1$, we have that the separation is at most
	\begin{align*}
		g(\yv)  
		&= \frac{1}{n} \sum_{i=1}^n \sum_{e_0 \in \expset{1}^0} 
		\sup \braces[\Big]{ (\beta_{i,e_1} - \beta_{i,e_1'} )^2 : e_1, e_1' \in  \splitset{i}{e_0} }
		\\
		&\leq \frac{1}{n} \sum_{i=1}^n \sum_{e_0 \in \expset{i}^0}  4 \cdot \indicator{ \abs{\splitset{i}{e_0}} > 1 } \\
		&\triangleq 4 \cdot \avgsplit
	\end{align*}
	Thus, it follows that $\mix{1}$ is supported on $\althypd$ for every non-trivial separation value $\delta$.
\end{proof}

Recall that $\marg{k}$ is the distribution of the observed potential outcomes $\yv(\zv) = (y_1(\zv), y_2(\zv), \dots y_n(\zv))$ given the intervention $\zv$ under mixture $\mix{k}$.
The following lemma shows that for every intervention, the two distributions $\marg{0}$ and $\marg{1}$ are indistinguishable.

\begin{lemma} \label{lemma:mixture-separate}
	For each intervention $\zv \in \allinterventions$, $\tvdist{\marg{0}}{\marg{1}} = 0$.
\end{lemma}
\begin{proof}
	Fix an arbitrary intervention $\zv \in \allinterventions$.
	For each subject $i \in [n]$, let $\expmap{i}^0(\zv) = e_{i,0}$ be the coarse exposure realized by intervention $\zv$.
	Similarly, let $\expmap{i}^1(\zv) = e_{i,1}$ be the finer realized exposure.
	If $\yv \sim \mix{0}$, then
	\[
	\yv(\zv) = ( \alpha_{1, e_{1,0}}, \dots \alpha_{n, e_{n,0}}  )
	\]
	has independent entries with are $\pm 1$ with equal probability.
	Likewise, if $\yv \sim \mix{1}$ then 
	\[
	\yv(\zv) = ( \beta_{1, e_{1,1}}, \dots \beta_{n, e_{n,1}}  )
	\]
	has independent entries which are also $\pm 1$ with equality probability.
	Thus, it follows that the two marginal distributions $\marg{0}$ and $\marg{1}$ are equal and thus $\tvdist{\marg{0}}{\marg{1}} = 0$.
\end{proof}

\begin{proof}[\ifnum \value{journal}=1 {} \else {Proof} \fi of Theorem~\mainref{theorem:general-impossibility}]
	Fix a non-trivial separation value $\delta$ and consider the null and well-separated alternative models $\nullhyp$ and $\althypd$.
	Consider the mixtures $\mix{0}$ and $\mix{1}$ defined as above.
	By Lemma~\ref{lemma:mixture-support}, we have that $\mix{0}$ and $\mix{1}$ are supported on $\nullhyp$ and $\althypd$, respectively.
	Thus, we apply Lemma~\mainref{lemma:mixture-lower-bound} together with Lemma~\ref{lemma:mixture-separate} to obtain that for any design $\design$,
	\[
	\optrisk(\delta) 
	\geq 1 - \sup_{\design} \Esub[\big]{\zv \sim \design}{ \tvdist{\marg{0}}{\marg{1}} }
	= 1 - \sup_{\design} \Esub[\big]{\zv \sim \design}{0}
	= 1 
	\enspace,
	\]
	which completes the proof.
\end{proof}

	\section{Consistent Test Against Linear-in-Means Model} \label{sec:supp-lim-consistency}

In this section, we provide a detailed analysis for the specification test for SUTVA against a network linear-in-means model.
We begin by reviewing the setting.
The null model is given as
\[
\nullmodel = \setb[\big]{ \yv :
	 y_i(\zv) =  \alpha_{i,1} \cdot \indicator{z_i = 1} + \alpha_{i,0} \cdot \indicator{z_i = 0} 
	 \text{ and }
	 \norm{\yv}_{\infty} \leq 1
 }
\enspace
\]
and the alternative is given as
\[
\altmodel = \setb[\Big]{ \yv :
	 y_i(\zv) =  \beta_{i,1}+ \beta_{i,2} z_i + \beta_{i,3} \paren[\Big]{ \frac{1}{\neighi} \sum_{j \in \neighi} z_j }
	 \text{ and }
	 \norm{\yv}_{\infty} \leq 1
 }
\enspace,
\]
where $\neighi$ is the neighborhood of subject $i$ in the underlying network $G$.
Recall that the degree of each node is denoted $\degreei = \abs{ \neighi }$ and the maximum degree is $\dmax = \max_{i \in [n]} \degreei$.

In the main body, we made the simplifying assumption that all subjects in the network have degree at least $\degreei \geq 2$, but we consider the general case here.
We still assume that $\degreei \geq 1$ for all $i \in [n]$, as the parameterization $(\beta_{i,1}, \beta_{i,2}, \beta_{i,3})$ of $\altmodel$ is not injective when $\degreei = 0$, so a more intricate separation functional would be needed to accommodate units with $\degreei = 0$.
The restriction $\degreei \geq 1$ is merely a technicality, because the null and alternative models trivially coincide for units with $\degreei = 0$, so the null is known to be true a priori for units with no neighbors. 

Recall that the separation functional is defined as
\[
g(\yv) = \frac{1}{n} \sum_{i=1}^n \beta_{i,3}^2
\enspace.
\]
Our goal is to construct a consistent estimator of the separation value $g(\yv)$ for any $\yv \in \altmodel$ with uniformly bounded outcomes $\norm{\yv}_{\infty} \leq 1$.
We will do so by constructing individual estimators for the $\beta_{i,3}^2$ terms.

To this end, we propose the following weighted estimator.
We designed this weighting using the general Riesz representation principle for quadratic functionals provided in \citet{Harshaw2022Design}.
We omit the tedious calculations required to derive the estimator using this technique; instead, the following proposition demonstrates the feasibility of the general approach:

\begin{proposition} \label{prop:general-unbiased-weights}
	Let $Z$ and $X$ be independent random variables.
	Define the variable
	\[
	W = \frac{\Var{X}\paren[\big]{ X^2 - \E{X^2} } - \paren[\big]{ \E{X^3} - \E{X^2} \E{X}  } \paren[\big]{ X - \E{X} } }{ \Var{X} \paren{\E{X^4} - \E{X^2}^2 } - \paren{ \E{X^3} - \E{X^2} \E{X} }^2 }
	\enspace.
	\]
	If the denominator is non-zero, then $W$ is well-defined and
	\[
	\E[\Big]{ \braces{ \beta_{1} + \beta_{2} Z + \beta_{3} X }^2 \cdot W } = \beta_{3}^2
	\enspace.
	\]
\end{proposition}
\begin{proof}
	First, let us expand the square and use linearity of expectation
	\begin{align*}
		&\E[\Big]{ \braces{ \beta_{1} + \beta_{2} Z + \beta_{3} X }^2 \cdot W } \\
		&\quad = 
		\E[\Big]{ \braces[\big]{ \beta_1 + \beta_2 \cdot Z }^2 W }
		+ 2 \beta_3 \E[\Big]{ \braces[\big]{ \beta_1 + \beta_2 \cdot Z } X W }
		+ \beta_3^2 \E[\Big]{ X^2 W }
		\intertext{
			Recall that $X$ and $Z$ are independent random variables and that $W$ is a function of $X$. Thus, $W$ is also independent of $Z$.
			Using this to factor the expectations, we get that
		}
	&\quad = 
	\E[\Big]{ \paren[\big]{\beta_1 + \beta_2 \cdot Z}^2 } \E{ W }
	+ 2 \beta_3 \E[\Big]{ \braces[\big]{ \beta_1 + \beta_2 \cdot Z } } \E[\big]{ X W }
	+ \beta_3^2 \E[\big]{ X^2 W }
	\end{align*}
	To complete the proof, it suffices to show that $\E{W} = 0$, $\E{XW} = 0$, and $\E{X^2 W} = 1$.
	To this end, let us calculate each of these quantities:
	\begin{align*}
		\E{W} 
		&= \frac{ \Var{X} \braces[\big]{ \E{X^2} - \E{X^2} } + \paren[\big]{\E{X^3} - \E{X^2} \E{X} } \braces[\big]{ \E{X} - \E{X} } }{\Var{X} \paren{\E{X^4} - \E{X^2}^2 } - \paren{ \E{X^3} - \E{X^2} \E{X} }^2 }
		=0 \\
		\E{XW}
		&= \frac{ \Var{X} \braces[\big]{ \E{X^3} - \E{X^2}\E{X} } + \paren[\big]{\E{X^3} - \E{X^2} \E{X} } \braces[\big]{ \E{X^2} - \E{X}^2 }}{\Var{X} \paren{\E{X^4} - \E{X^2}^2 } - \paren{ \E{X^3} - \E{X^2} \E{X} }^2 }
		=0 \\
		\E{X^2 W}
		&= \frac{ \Var{X} \braces[\big]{ \E{X^4} - \E{X^2}^2 } + \paren[\big]{E{X^3} - \E{X^2} \E{X} } \braces[\big]{ \E{X^3} - \E{X^2} \E{X} }}{\Var{X} \paren{\E{X^4} - \E{X^2}^2 } - \paren{ \E{X^3} - \E{X^2} \E{X} }^2 }
		=1
		\enspace,
	\end{align*}
	which completes the proof.
\end{proof}

We apply this general weighting estimator to the setting of the linear-in-means model.
For notational simplicity, we define the (random) fraction of treated neighbors as $x_i = 1 / \degreei \sum_{j \in \neighi} z_j$.
Throughout this section, we will focus on the Bernoulli design where binary treatments $z_1 \dots z_n$ are chosen independently with $\Pr{z_i = 1} = p$.

Our next goal will to be calculate the first four moments of the random fraction of treated neighbors $x_i$.
In these calculations, the following fact will be very useful: the $k$th moment of a Bernoulli variable $X$ with parameters $n$, $p$ is given by
\[
\E{ X^k } = \sum_{\ell=0}^k { n \brack k } n^{\underline{k}} p^k \enspace,
\]
where ${n \brack k}$ is the Stirling number of of the second kind and $n^{\underline{k}}$ is the $k$-th falling power of $n$.
This will be useful because under the Bernoulli design, the number of treated neighbors $\degreei \cdot x_i$ is Binomial with parameters $\degreei$ and $p$.
Using this, we can calculate the first four moments of $x_i$ as
\begin{enumerate}
	\item $\degreei \cdot \E{x_i} = \degreei p$
	\item $\degreei^2 \cdot \E{x_i^2} = \degreei p + \degreei(\degreei-1) p^2$
	\item $\degreei^3 \cdot \E{ x_i^3 } = \degreei p + 3 \degreei (\degreei - 1) p^2 + \degreei (\degreei-1)(\degreei-2) p^3$
	\item $\degreei^4 \cdot \E{ x_i^4 } = \degreei p + 7 \degreei (\degreei - 1) p^2 + 6 \degreei (\degreei-1)(\degreei-2) p^3 + \degreei(\degreei-1)(\degreei-2)(\degreei-3) p^4$
\end{enumerate}
Now that the first four moments of $x_i$ are calculated, we can express the weighting term $W_i$ in terms of the treatment probability $p$ and the degree $\degreei$ as
\begin{align*}
W_i 
&=
\frac{\Var{x_i}\paren[\big]{ x_i^2 - \E{x_i^2} } - \paren[\big]{ \E{x_i^3} - \E{x_i^2} \E{x_i}  } \paren[\Big]{ x_i - \E{x_i} } }{ \Var{x_i} \paren{\E{x_i^4} - \E{x_i^2}^2 } - \paren{ \E{x_i^3} - \E{x_i^2} \E{x_i} }^2 } \\
&= 
\frac{\degreei^2}{2 \degreei (\degreei - 1) p^2 (1-p)^2}
 \braces[\Bigg]{ 
	\degreei^2 \paren[\big]{ x_i^2 - \E{x_i^2} }
	- \braces{ 2p \degreei ( \degreei-1 ) + \degreei } \paren[\big]{ x_i - \E{x_i} } 
}
\enspace.
\end{align*}
For the remainder of the appendix, we restrict our attention to an equal treatment probability $p = 1/2$.
This choice greatly simplifies the weighting term:
\[
W_i = 
8 \degreei^2 \paren[\big]{ 1 - \frac{1}{\degreei} }^{-1}
\braces[\Big]{ \paren[\big]{ x_i^2 - \E{x_i^2} } - \paren[\big]{ x_i - \E{x_i} } } 
\enspace.
\]
Note that this weighting term is well-defined only when the subject has degree at least $\degreei \geq 2$.
In other words, a subject is required to have at least 2 neighbors in order to ensure sufficient variation in $x_i$ so that our weighting technique can estimate $\alpha_{i,3}^2$ in an unbiased manner.
Let $T = \setb{ i : \degreei = 1 }$ denote the set of subjects with exactly one neighbor.

We define the separation estimator as
\[
\sepest = \frac{1}{n} \sum_{i \notin T}^n Y_i^2 \cdot W_i
\enspace.
\]
The following proposition characterizes the bias of the separation estimator.

\begin{proposition} \label{prop:sep-est-bias}
	For all $\yv \in \altmodel$, the bias of the separation estimator is bounded as
	\[
	g(\yv) - \E{\sepest}
	= \frac{1}{n} \sum_{i \in T} \beta_{i,3}^2
	\leq 4 \cdot \frac{\abs{T}}{n}
	\enspace.
	\]
\end{proposition}
\begin{proof}
	By Proposition~\ref{prop:general-unbiased-weights}, we have that $\E{ \sepest } = \frac{1}{n} \sum_{i \notin T} \beta_{i,3}^2$.
	Thus,
	$
	g(\yv) - \E{ \sepest }
	 = \frac{1}{n} \sum_{i \in T} \beta_{i,3}^2
	$.
	To finish the proof, it suffices to show that $\abs{\beta_{i,3}} \leq 2$.
	To this end, observe that $\beta_{i,3} = y_i(\onevec - \vec{e}_i) - y_i(\zerovec)$.
	Using the triangle inequality together with the uniform bound $\norm{\yv}_{\infty} \leq 1$ yields
	\[
	\abs{\beta_{i,3}}
	= \abs{ y_i(\onevec - \vec{e}_i) - y_i(\zerovec) }
	\leq \abs{ y_i(\onevec - \vec{e}_i)} + \abs{ y_i(\zerovec) }
	\leq 2
	\enspace,
	\]
	which completes the proof.
\end{proof}

We now seek to bound the variance of our estimator.
We will use the dependency graph method where first we calculate the variance of the individual estimators, and then we bound the number of individual estimators that can be correlated.
The technique itself is somewhat crude as the covariance between individual estimators is bounded by the Cauchy--Schwarz inequality.
Although it is possible to refine this method further, this analysis suffices for our present goal.

\begin{lemma} \label{lemma:individual-estimators-var}
	Consider the Bernoulli design with $p = 1/2$.
	For all $\yv \in \altmodel$,
	the variance of the individual estimator for subject $i \notin T$ is bounded as
	\[
	\Var{Y_i^2 \cdot W_i} \leq \norm{\yv}_{\infty} \cdot 2^9 \cdot \degreei^3 
	\enspace.
	\]
\end{lemma}
\begin{proof}
	We begin by bonding the variance by the second moment and using the uniform bound on the outcomes:
	\[
	\Var{ Y_i^2 W_i }
	\leq \E[\Big]{ Y_i^4 W_i^2 }
	\leq \norm{\yv}_{\infty}^4 \cdot \E{ W_i^2 }
	\enspace.
	\]
	Using the formula for the weight $W_i$ when $p = 1/2$ together with Cauchy-Schwarz, and AM-GM we obtain
	\begin{align}
	\E{ W_i^2 }
	&= 2^6 \degreei^4 \paren[\Big]{ 1 - \frac{1}{d_i} }^{-2}
	\E[\Big]{ \braces[\Big]{ \paren{ x_i^2 - \E{ x_i^2 } } - \paren{ x_i - \E{ x_i } }  }^2 } \\
	& \leq 2^7 \degreei^4 \paren[\Big]{ 1 - \frac{1}{d_i} }^{-2}
	\braces[\Big]{ \E[\big]{ \paren{ x_i^2 - \E{x_i^2} }^2 } + \E[\big]{ \paren{ x_i - \E{x_i} }^2 } } \\
	&\leq 2^9 \degreei^4
	\braces[\Big]{ \E[\big]{ \paren{ x_i^2 - \E{x_i^2} }^2 } + \E[\big]{ \paren{ x_i - \E{x_i} }^2 } }
	\enspace,
	\end{align}
	where in the last line we have used that $\degreei \geq 2$ because $i \notin T$.
	We now seek to bound the two expectations.
	First, we calculate 
	\[
	\E[\big]{ \paren{ x_i - \E{x_i} }^2 } 
	= \Var{ x_i }
	= \Var[\Big]{ \frac{1}{\degreei} \sum_{j \in \neighi} z_j }
	= \frac{1}{\degreei} \sum_{j \in \neighi} \Var{z_j}
	= \frac{p(1-p)}{\degreei}
	= \frac{1}{4 \degreei}
	\enspace.
	\]
	Next, using the expressions for $\E{ x_i^4 }$ and $\E{ x_i^2 }$ provided above and going through the algebra, we obtain that
	\begin{align*}
		\degreei^4 \cdot \E{ (x_i^2 - \E{ x_i^2 })^2 }
		&= \degreei^4 \cdot \E{ x_i^4 } - \paren[\Big]{ \degreei \E{ x_i^2 } }^2 \\
		&= \degreei \paren[\big]{ p - 7p^2 + 12 p^3 - 6 p^4 } 
		+ \degreei^2 \paren[\big]{ 6 p^2 - 16 p^3 + 10p^4 }
		+ \degreei^3 \paren[\big]{ 4 p^3 - 4p^4 }
	\end{align*}
	Setting $p = 1/2$ and simplifying further we obtain
	\[
		\E[\big]{ (x_i^2 - \E{ x_i^2 })^2 }
		= \frac{1}{\degreei^4} \braces[\Big]{ - \frac{1}{2^3} \degreei + \frac{1}{2^3} \degreei^2 + \frac{1}{2^2} \degreei^3 } \\
		= \frac{1}{2^3} \braces[\Big]{ \frac{2}{\degreei} + \paren[\big]{ 1 - \frac{1}{\degreei} } \frac{1}{\degreei^2} } \\
		\leq \frac{1}{2^2} \braces{\frac{1}{\degreei} + \frac{1}{\degreei^2} } \\
		\leq \frac{1}{2 \degreei}
		\enspace,
	\]
	where in the final two inequalities used that $\degreei \geq 2$ and $\degreei^2 \geq \degreei$.
	Plugging these calculations back into the upper bound for $\E{ W_i^2 }$, we have that
	\[
	\E{ W_i^2 }
	\leq 2^9 \degreei^4 \paren[\Big]{ \frac{1}{2 \degreei} + \frac{1}{4 \degreei} }
	\leq 2^9 \degreei^3
	\enspace,
	\]
	which completes the proof.
\end{proof}

The following lemma keeps track of which individual estimators are independent.

\begin{lemma} \label{lemma:individual-estimators-independent}
	If $\neighi \cap \neigh{j} = \emptyset$, then the individual estimators $Y_i^2 W_i$ and $Y_j^2 W_j$ are independent.
\end{lemma}
\begin{proof}
	$Y_i^2 W_i$ is a function of the variables $z_k \in \neigh{i}$ and $Y_j^2 W_j$ is a function of the variables $z_\ell \in \neigh{j}$.
	If $\neighi \cap \neigh{j} = \emptyset$ then these sets of variables are disjoint.
	Under the Bernoulli design, all of the treatments $z_1 \dots z_n$ are jointly independent.
	Thus, the individual estimators $Y_i^2 W_i$ and $Y_j^2 W_j$ inherent this independence.
\end{proof}

Finally, we use the dependency graph method to bound the variance of the separation estimator.

\begin{proposition}\label{prop:sep-est-var-bound}
	For all $\yv \in \altmodel$, the separation estimator is bounded as
	\[
	\Var{\sepest} \leq 2^9 \cdot \frac{\dmax^5}{n}
	\enspace.
	\]
\end{proposition}
\begin{proof}
	For notational simplicity, we define $W_i = 0$ for $i \in T$.
	By expanding the variance and using the independence properties of Lemma~\ref{lemma:individual-estimators-independent}, we have that
	\begin{align*}
		\Var{\sepest}
		&= \Var[\Big]{ \frac{1}{n} \sum_{i=1}^n Y_i^2 W_i } \\
		&= \frac{1}{n^2} \sum_{i=1}^n \sum_{j=1}^n \Cov{ Y_i^2 W_i, Y_j^2 W_j } 
		\intertext{
			Next, we kill the covariance terms which are known to be zero by the independence of Lemma~\ref{lemma:individual-estimators-independent}.
			Leting $M_i = \setb{j : \neighi \cap \neigh{j} \neq \emptyset}$, we have}
		&= \frac{1}{n^2} \sum_{i=1}^n \sum_{j \in M_i} \Cov{ Y_i^2 W_i, Y_j^2 W_j }  \\
		\intertext{Using the Cauchy-Schwarz inequality together with the bound on the individual variances from Lemma~\ref{lemma:individual-estimators-var}, we have that}
		&\leq \frac{1}{n^2} \sum_{i=1}^n \sum_{j \in M_i} \sqrt{ \Var{ Y_i^2 W_i } \Var{ Y_j^2 WIj } } \\
		&\leq \frac{2^9  }{n^2} \sum_{i=1}^n \sum_{j \in M_i} \degreei^{3/2} \degreej^{3/2} \\
		&= \frac{2^9  }{n^2} \sum_{i=1}^n \degreei^{3/2} \sum_{j \in M_i} \degreej^{3/2}
		\intertext{At this point, we remark that $\abs{M_i} \leq \degreei \cdot \dmax$ because each subject $i$ has $\degreei$ neighbors which each have at most $\dmax$ neighbors. Likewise, we can bound $\degreej \leq \dmax$. Using these bounds, we obtain}
		&\leq \frac{2^9  }{n^2} \sum_{i=1}^n \degreei^{3/2} \cdot \degreei \dmax \cdot \dmax^{3/2} \\
		&= 2^9  \cdot \frac{\dmax^{5/2}}{n} \cdot \paren[\Big]{ \frac{1}{n} \sum_{i=1}^n \degreei^{5/2} }  \\
		&\leq 2^9 \cdot  \frac{\dmax^{5}}{n}  
		\enspace,
	\end{align*}
	where the last line used that $\degreei \leq \dmax$.
\end{proof}

Inspecting the proof above reveals that we can obtain somewhat stronger bounds than $\dmax^5 / n$. 
For example, we could obtain results that have a weaker dependence on $\dmax$ and instead use the $2.5$-th moment of the degree distribution.
We do not explore such improvements in this manuscript since obtaining a tight analysis of the rate is not our primary concern.

Combining the analysis of the bias (Proposition~\ref{prop:sep-est-bias}) and variance (Proposition~\ref{prop:sep-est-var-bound}) yield the following result on the mean squared error of the separation estimator.

\renewcommand\theproposition{\mainref{prop:sep-estimator-mse}*}
\begin{proposition}
	For all $\yv \in \altmodel$,
	the mean square error of the separation estimator is bounded as
	\[
	\E[\big]{ \paren{ g(\yv) - \sepest }^2 }^{1/2}
	\leq 2^{5} \cdot  \paren[\Big]{ \frac{\abs{T}}{n} \vee \frac{\dmax^{5/2}}{n^{1/2}} }
	\enspace.
	\]
\end{proposition}
\renewcommand\theproposition{\text{S}\arabic{proposition}}
\addtocounter{proposition}{-1}

In order for the separation estimator to be consistent, we suppose that the graph satisfies two regularity conditions.
The first is that $\abs{T} = \littleO{n}$, which is to say that a vanishingly small fraction of nodes have fewer than 2 neighbors.
The second is that $\dmax = \littleO{n^{1/5}}$ which limits the asymptotic growth of the largest neighborhood.

Recall that our statistical test is based on evaluating whether the estimated separation $\sepest$ lies above or below a threshold $\tau$, i.e. $\test = \indicator{\sepest \geq \tau}$.
In order to ensure that the threshold goes to zero but is also asymptotically dominating the mean square error, we set it as
$
\tau = \paren{ \abs{T} / n \vee \dmax^{5/2} / n^{1/2} }^{1/2}
$.
The following theorem uses Markov's inequality to show that this is a consistent test.

\renewcommand\thetheorem{\mainref{theorem:lim-consistency}*}
\begin{theorem}
	Suppose that the network $G$ satisfies the regularity conditions $\abs{T} = \littleO{n}$ and $\dmax = \littleO{n^{1/5}}$.
	The Bernoulli design together with the statistical test $\test = \indicator{\sepest \geq \tau}$ with threshold $\tau = \paren{ \abs{T} / n \vee \dmax^{5/2} / n^{1/2} }^{1/2}$ is uniformly consistent: $\lim_{n \to \infty} \risk(\design, \test, \delta) = 0$ for all $\delta > 0$.
\end{theorem}
\renewcommand\thetheorem{\text{S}\arabic{theorem}}
\addtocounter{theorem}{-1}

\begin{proof}
	Let us being by considering a fixed separation value $\delta > 0$.
	We seek to show that the testing error converges to zero: $\lim_{n \to \infty} \risk(\design, \test, \delta) = 0$.
	
	Consider a function in the null model $\yv \in \nullmodel$.
	Using Markov's inequality, the Type I error can be bounded as
	\begin{align*}
		\E[\Big]{ \test(\zv, \yv(\zv)) }
		&= \Pr{ \sepest \geq \tau } \\
		& = \Pr{ \sepest - g(\yv) \geq \tau }
			&\text{($g(\yv) = 0$ because $\yv \in \nullmodel$)} \\
		&\leq \Pr{ \paren{ \sepest - g(\yv) }^2 \geq \tau^2 } \\
		&\leq \frac{ \E[\Big]{ \paren[\big]{ \sepest - g(\yv) }^2 } }{\tau^2}
			&\text{(Markov's inequality)} \\
		&\leq \frac{ 2^{10} \braces[\Big]{ \frac{\abs{T}}{n} \vee \frac{\dmax^{5/2}}{n^{1/2}} }^2 }{ \braces[\Big]{ \frac{\abs{T}}{n} \vee \frac{\dmax^{5/2}}{n^{1/2}} } } 
			&\text{(Proposition~\mainref{prop:sep-estimator-mse}* and $\tau$ choice)}\\
		&= 2^{10}  \braces[\Big]{ \frac{\abs{T}}{n} \vee \frac{\dmax^{5/2}}{n^{1/2}} }
	\end{align*}
	Thus, the supremum of Type I error over the null model can be bounded as
	\[
	\sup_{\yv \in \nullhyp} \E[\Big]{ \test(\zv, \yv(\zv))  } 
	\leq 2^{10} \braces[\Big]{ \frac{\abs{T}}{n} \vee \frac{\dmax^{5/2}}{n^{1/2}} }
	\enspace.
	\]
	
	Next, we bound the Type II error in the separated alternative $\althypd$.
	Under the graph regularity conditions and choice of threshold, we have that $\tau \to 0$.
	Suppose that $n$ is sufficiently large so that $\tau \leq \delta / 2$.
	In this case, we have that for any $\yv \in \althypd$, it follows that $g(\yv) \geq \delta \geq \tau+ \delta / 2$.
	The Type II error of the test for a function in the separated alternative $\yv \in \althypd$ is bounded as
	\begin{align*}
	\E[\Big]{ 1 - \test(\zv, \yv(\zv))  }
	&= \Pr[\Big]{ \sepest < \tau } \\
	&= \Pr[\Big]{ g(\yv) - \sepest > g(\yv) - \tau } \\
	&\leq \Pr[\Big]{ g(\yv) - \sepest > \delta/2 }
		&\text{($g(\yv) \geq \tau+ \delta / 2$)} \\
	&\leq \Pr[\Big]{ \paren{ g(\yv) - \sepest}^2 > (\delta/2)^2 } \\
	&\leq \frac{4}{\delta^2} \cdot \E[\Big]{ \paren[\big]{ \sepest - g(\yv) }^2 } 
		&\text{(Markov's inequality)}\\
	&\leq \frac{4}{\delta^2} 2^{10} \paren[\Big]{ \frac{\abs{T}}{n} \vee \frac{\dmax^{5/2}}{n^{1/2}} }^2
	\end{align*}
	Using these together, we have that the limiting testing error converges to 0 as
	\begin{align*}
	\lim_{n \to \infty} \risk(\design, \test, \delta)
	&= \lim_{n \to \infty} \sup_{\yv \in \nullhyp} \E[\Big]{ \test(\zv, \yv(\zv)) }
	+ \sum_{\yv \in \althypd} \E[\Big]{ 1 - \test(\zv, \yv(\zv)) } \\
	&\leq \lim_{n \to \infty} 
	2^{10} \braces[\Big]{ \frac{\abs{T}}{n} \vee \frac{\dmax^{5/2}}{n^{1/2}} }
	+
	\frac{4}{\delta^2} 2^{10} \paren[\Big]{ \frac{\abs{T}}{n} \vee \frac{\dmax^{5/2}}{n^{1/2}} }^2 \\
	&= 0
	\enspace,
	\end{align*}
	which completes the proof.
\end{proof}

\end{document}